\documentclass{article}
\usepackage{fancyheadings,amsmath,amssymb,amsthm}
\usepackage{amsfonts}
\usepackage{enumerate}
\usepackage[nodisplayskipstretch]{setspace}
\usepackage{titlesec}
\usepackage{graphicx}
\usepackage[section,above]{placeins}
\usepackage{xcolor}
\usepackage{color}
\usepackage[colorlinks,citecolor=blue,linkcolor=blue,pagecolor=blue,urlcolor=black]{hyperref}
\usepackage{doi}
\usepackage{float,subfloat}
\usepackage[caption=false]{subfig}
\usepackage{caption}
\usepackage{hypcap,breakurl}
\usepackage{cite}
\usepackage[square,numbers,sort&compress]{natbib}
\usepackage{setspace,geometry}
\usepackage{verbatim}

\usepackage{mpagM}
\titlelabel{\thetitle.\quad}

\DeclareMathOperator{\grad}{grad}

\begin{document}
\title{Lagrange stability of semilinear differential-algebraic equations and application to nonlinear electrical circuits}

\author{Maria S. Filipkovska}

\address{B.~Verkin Institute for Low Temperature Physics and  Engineering of the National Academy of Sciences of Ukraine\\
\scriptsize pr. Nauky 47, 61103 Kharkiv, Ukraine\\
V.N. Karazin Kharkiv National University\\
\scriptsize Svobody Sq. 4, 61022 Kharkiv, Ukraine\\
\smallskip {\rm E-mail: filipkovskaya@ilt.kharkov.ua}}

%\udc{}% Required parameter

\date{} % Date of receipt of the article

%%%%%%%%%%%%%%%%%%%%%%%%
\protect\maketitle \markright{Lagrange stability and instability of semilinear DAEs and applications}% If the full title of the article is %long
%%%%%%%%%%%%%%%%%%%%%%%% Enviroments %%%%%%%
\renewenvironment{proof}{\begin{trivlist}\item[]
{\quad\, P r o o f. }}{\hfill\rule{0.5em}{0.5em}\end{trivlist}}
\newtheorem{theorem}{\,\quad Theorem}             
\newtheorem{lemma}{\,\quad Lemma}                  
\newtheorem{definition}{\,\quad Definition}       
\newtheorem{corollary}{\,\quad Corollary}          
\newtheorem{proposition}{\,\quad Proposition}      
\newtheorem{example}{\,\quad\rm E x a m p l e }   
\newtheorem{remark}{\,\quad\rm R e m a r k }       

\newcommand{\R}{{\mathbb R}}
\newcommand{\Rn}{{\mathbb R}^n}
\newcommand{\Rm}{{\mathbb R}^m}
\newcommand{\Cn}{{\mathbb C}^n}
\newcommand{\Cm}{{\mathbb C}^m}
\newcommand{\Ra}{{\mathbb R}^a}

\newcommand*{\hm}[1]{#1\nobreak\discretionary{}%
{\hbox{$\mathsurround=0pt #1$}}{}}

\setlength{\textfloatsep}{5pt plus 2.0pt minus 2.0pt}
\setlength{\floatsep}{1pt plus 1.0pt minus 1.0pt}
\setlength{\intextsep}{5pt plus 2.0pt minus 2.0pt}

\renewcommand{\thefigure}{\thesection.\arabic{figure}}
\renewcommand{\thetheorem}{\thesection.\arabic{theorem}}
\renewcommand{\thedefinition}{\thesection.\arabic{definition}}
\renewcommand{\theequation}{\thesection.\arabic{equation}}
\renewcommand{\theremark}{\thesection.\arabic{remark}}
\renewcommand{\theexample}{\thesection.\arabic{example}}

\numberwithin{figure}{section}
\numberwithin{theorem}{section}
\numberwithin{definition}{section}
\numberwithin{remark}{section}
\numberwithin{example}{section}
\numberwithin{equation}{section}

\begin{abstract}
We study a semilinear differential-algebraic equation (DAE) with the focus on the Lagrange stability (instability). The conditions for the existence and uniqueness of global solutions (a solution exists on an infinite interval) of the Cauchy problem, as well as conditions of the boundedness of the global solutions, are obtained. \linebreak
Furthermore, the obtained conditions for the Lagrange stability of the semilinear DAE guarantee that every its solution is global and bounded, and, in contrast to theorems on the Lyapunov stability, allow to prove the existence and uniqueness of global solutions regardless of the presence and the number of equilibrium points. We also obtain the conditions of the existence and uniqueness of solutions with a finite escape time (a solution exists on a finite interval and is unbounded, i.e., is Lagrange unstable) for the Cauchy problem. We do not use constraints of a global Lipschitz condition type, that allows to use the work results efficiently in practical applications. The mathematical model of a radio engineering filter with nonlinear elements is studied as an application. The numerical analysis of the model verifies the results of theoretical investigations.

\smallskip


\end{abstract}\smallskip

\section{Introduction}\label{Intro}

Differential-algebraic equations (DAEs), which are also called descriptor, algebraic-dif\-feren\-tial and degenerate differential equations, have a wide range of practical applications. Certain classes of mathematical models in radioelectronics, control theory, economics, %\linebreak
robotics technology, mechanics and chemical kinetics are described by semilinear DAEs. Semilinear DAEs comprise in particular semiexplicit DAEs and in turn can be attributed to quasilinear DAEs.  The Lagrange stability of a DAE guarantees that every its solution is global and bounded. The presence of a global solution of the equation guarantees a sufficiently long action term of the corresponding real system. The properties of boundedness and stability of solutions of the equations describing mathematical models are used in the design and synthesis of the corresponding real systems and processes. The application of DAE theory to a study of electrical circuits is presented in various monographs and papers, for example, in \cite{Kunkel_Mehrmann, Dai, Lamour-Marz-Tisch., Riaza, Brenan-C-P, Riaza-Marz, Rabier-Rheinboldt, Reis-Stykel, Rut, OMalley-Kalachev} (see also references in them).

In the present paper, the semilinear differential-algebraic equation (DAE)
\begin{equation}\label{DAE}
 \frac{d}{dt}[Ax]+Bx=f(t,x)
\end{equation}
with a nonlinear function $f\colon [0,\infty)\times \Rn \to \Rn$ and linear operators $A,\, B\colon \Rn \to \Rn$ is considered. The operator $A$ is degenerate (noninvertible), the operator $B$ may also be degenerate. Note that solutions of a semilinear DAE of the form $\displaystyle A\frac{d}{dt}x+Bx=f(t,x)$  must be smoother than solutions of a semilinear DAE of the form \eqref{DAE}.  The availability of a noninvertible operator (matrix) at the derivative in the DAE means the presence of algebraic connections, which influence the trajectories of solutions and impose restrictions on the initial data.
For the DAE \eqref{DAE} with the initial condition
\begin{equation}\label{ini}
 x(t_0)=x_0,
\end{equation}
the initial value $x_0$ must be chosen so that the initial point $(t_0,x_0)$ belongs to the manifold $L_0 = \{(t,x)\! \in\!  [0,\infty) \times  \Rn  \mid  Q_2[Bx-f(t,x)] = 0\}$ (which is also defined in \eqref{soglreg2}, where $Q_2$ is a spectral projector considered in Section~\ref{Prelim}). The initial value $x_0$ satisfying the consistency condition $(t_0,x_0)\in L_0$ is called a consistent initial value.

The influence of the linear part $\frac{d}{dt}[Ax]+Bx$ of the DAE \eqref{DAE} is determined by the properties of the pencil $\lambda A+B$ ($\lambda$ is a complex parameter). It is assumed that $\lambda A+B$ is a regular pencil of index~1, i.e., there exists the resolvent of the pencil $(\lambda A+B)^{-1}$ and it is bounded for sufficiently large $|\lambda |$ (see Section~\ref{Prelim}).
This property of the pencil allows to use the spectral projectors $P_1$, $P_2$, $Q_1$, $Q_2$ which can be calculated by contour integration and reduce the DAE to the equivalent system of a purely differential equation and a purely algebraic equation (see Section~\ref{Prelim}). This is one of the reasons why we use the requirement of index~1 for the characteristic pencil $\lambda A+B$ of the linear part of the DAE and not for the DAE, as, for example, in \cite{Marz,Tischendorf,Lamour-Marz-Tisch.}. Another reason is as follows. The requirement that the DAE has index~1 does not give us the necessary result and is too restrictive for our research (this will be discussed in Section~\ref{Prelim}). It is also worth noting that semilinear DAEs of the form \eqref{DAE} arise in many practical problems, examples of which can be found in the books and papers by R.~Riaza, A.G.~Rutkas, A.~Favini, L.A.~Vlasenko,  A.D.~Myshkis, S.L.~Campbell, L.R.~Petzold,  K.E.~Brenan,  E.~Hairer, G.~Wanner, J.~Huang, J.F.~Zhang,  R.E.~Showalter and other authors (however, in present literature these equations are often written in the form $\frac{d}{dt}[Ax]=g(t,x)$ or in the form of semiexplicit DAE).

The objective of the paper is to find the conditions of the Lagrange stability and instability of the semilinear DAE (see Definitions~\ref{Def-FiniteTime}--\ref{Def-EqLStUnst}). A mathematical model of a radio engineering filter with nonlinear elements is considered as an application.  Note that if the operator $A$ in the semilinear DAE is invertible, then the results obtained in the paper remain valid (in this case the semilinear DAE is equivalent to an ordinary differential equation), however, we are interested in the case of the noninvertible operator.

In Section~\ref{SectThUst} the theorem on the Lagrange stability, which gives sufficient conditions of the existence and uniqueness of global solutions of the Cauchy problem for the semilinear DAE, as well as conditions of the boundedness of the global solutions, is proved. Furthermore, the theorem gives conditions of the Lagrange stability of the semilinear DAE, which ensure that each solution of the DAE starting at the time moment $t_0\! \in\! [0,\infty)$ exists on the whole infinite interval $[t_0,\infty)$ (is global) and is bounded. In Section~\ref{SectThNeust} the theorem on the Lagrange instability, which gives sufficient conditions of the existence and uniqueness of solutions with a finite escape time for the Cauchy problem, is proved.  It is important that the proved theorems do not contain restrictions of a global Lipschitz condition type, including the condition of contractivity, that allows to use them to solve more general classes of applied problems. Note that theorems on the unique global solvability of semilinear DAEs that comprise conditions equivalent to global Lipschitz conditions are known (cf. \cite{Vlasenko1}).
Also, the proved theorems do not contain the requirement that the DAE has index~1 globally (such requirement is contained, for example, in \cite[Thm.~6.7]{Lamour-Marz-Tisch.}). For comparison, the theorems from \cite{Marz, Tischendorf, Lamour-Marz-Tisch.} are considered in Section~\ref{Intro} below and in Section~\ref{Prelim}.

The investigation of the Lagrange stability of the ordinary differential equation (ODE) $\frac{d}{dt}x=f(t,x)$ ($t\ge 0$, $x$ is an $n$-dimensional vector) was made in \cite[Ch.~4]{LaSal-Lef-eng} using the method obtained by extensions of the direct (second) method of Lyapunov. The results of this investigation are made precise and are extended to semilinear DAEs in the present paper.  The existence and uniqueness theorem of a global solution of the Cauchy problem for the semilinear DAE with a singular pencil $\lambda A+B$ was proved in the author's paper \cite{Fil.sing}.  The results on the Lagrange stability of the semilinear DAE with the regular pencil, obtained by the author in \cite{Fil.num_meth}, have been improved and have been applied for a detailed study of evolutionary properties of the mathematical model for a radio engineering filter in the present paper.

The stability of linear DAEs and descriptor control systems described by linear DAEs, was studied by many authors  (see, for example, \cite{Dai, Riaza, Lamour-Marz-Tisch.,Shcheglova-Chistyakov} and references in them).

In \cite{Marz} R.~M\"arz investigated the Lyapunov stability of an equilibrium point of the autonomous ``quasilinear'' DAE
\begin{equation}\label{qDAE}
A\frac{d}{dt}x+g(x)=0,
\end{equation}
where $A\in L(\Rn)$ is singular (noninvertible) and $g\colon\! D\to \Rn$, $D\subseteq \Rn$ open.
The theorem \cite[Thm.~2.1]{Marz} \emph{allows to prove the existence and uniqueness of global solutions only in some} (sufficiently small) \emph{neighborhood of an equilibrium point $x^*$} of \eqref{qDAE}, i.e., $g(x^*)=0$, $x^*\!\in\! D$. If there are the two equilibrium points $x^*_1\! \in\! D$ and $x^*_2\! \in\! D$, $x^*_1\ne x^*_2$, then \emph{the theorem} %\linebreak
\cite[Thm.~2.1]{Marz} \emph{can not guarantee the existence of a unique global solution in~$D$}.  Namely, if the conditions of the theorem are fulfilled for the equilibria $x^*_1$ and $x^*_2$, then for some initial time moment $t_0$ there exist  the unique global solution $x=\varphi(t)$ of \eqref{qDAE} (with the initial condition \cite[(2.8)]{Marz}) in some neighborhood of $x^*_1$ and the unique global solution $x=\psi(t)$ of \eqref{qDAE} in some neighborhood of $x^*_2$, but this does not guarantee the existence of a unique global solution in $D$.\! \emph{Theorem~\ref{Th_Ust1} allows to prove the existence and uniqueness of global solutions for all possible initial points} (as noted in Remark~\ref{RemConsistIni}), that is,  \emph{regardless of the presence of an equilibrium point, in the presence of several equilibrium points or the infinite number of equilibrium points}, and for a more general equation than~\eqref{qDAE}.

A theorem similar to \cite[Thm.~2.1]{Marz} was proved by C.~Tischendorf \cite{Tischendorf} for the autonomous nonlinear DAE $f(x'(t),x(t))\!=\!0$. The theorem \cite[Thm.~3.3]{Tischendorf} gives conditions of the asymptotic stability (in Lyapunov's sense) of a stationary solution $x^*$, i.e., $f(0,x^*)\!=\!0$. The definition of asymptotic stability from \cite[Section~3]{Tischendorf} is equivalent to the fulfillment of the conditions (i)--(iii) from \cite[Thm.~2.1]{Marz}, and if we take $f(x'(t),x(t))\!=\!Ax'(t)+g(x(t))$, then \cite[Thm.~3.3]{Tischendorf} and \cite[Thm.~2.1]{Marz} will be analogous.

For a global solution of a nonautonomous nonlinear DAE conditions of the asymptotic stability (in Lyapunov's sense) which can also be considered just locally (in a sufficiently small neighborhood of this solution) are given in the theorem \cite[Thm.~6.16]{Lamour-Marz-Tisch.}. Under the conditions of the theorem it is assumed that the regular index-1 DAE has the global solution and a DAE linearized along this solution is strongly contractive \cite[Def.~6.5]{Lamour-Marz-Tisch.}.

It is important to note that the theorem on the Lagrange stability (Theorem~\ref{Th_Ust1}) gives conditions of the existence and uniqueness of global solutions (as well as conditions of the boundedness) independent of the presence and the number of equilibrium points. In contrast to Lyapunov stability, Lagrange stability can be considered as the stability of the entire system, not just of its equilibria.
From this, in particular, it follows that a globally stable dynamic system can be not only monostable (as in the case of the global stability in Lyapunov's sense), but also multistable (cf.~\cite[Section~I]{Wu-Zeng}). A.~Wu and Z.~Zeng \cite{Wu-Zeng} investigated the Lagrange stability of neural networks, which are described by ODEs with delay. It is known that neural networks are also described by DAEs (including semilinear DAEs), therefore the research carried out in the present paper is useful for the analysis and synthesis of the neural networks. Lagrange stability are also used for the analysis of ecological stability. The theorem on the Lagrange instability (Theorem~\ref{Th_Neust1}) can be used, in particular, for the analysis of nonlinear control systems. For example, the investigation of the Lagrange instability allows to find such property as the unboundedness of the response of a nonlinear control system on a finite time interval.

It is also important to note that even for an ordinary differential equation containing a nonlinear part the Lyapunov stability of a nontrivial solution does not imply that the solution is bounded, i.e., Lagrange stable.   Since the DAE considered in the paper contains the nonlinear part, then the Lyapunov stability of its solution does not imply the Lagrange stability. Also, in the general case, the Lyapunov instability does not imply  the Lagrange instability, but the converse assertion is true. Therefore, the proved Lagrange instability theorem can also be regarded as a Lyapunov instability theorem.

Thus, the research of the Lagrange stability of semilinear DAEs is of interest both to the DAE theory and to various applied problems.
The Lagrange stability of different types of ODEs and its applications are considered in many works, e.g., \cite{Wu-Zeng, LaSal-Lef-eng, Andrz-Awr, Bac-Lion}.
However, in \cite{Marz, Tischendorf, Lamour-Marz-Tisch.} and other cited works, the Lagrange stability of DAEs has not been researched.

In Sections~\ref{MatModUst} and~\ref{MatModNeust} the mathematical model of a radio engineering filter with nonlinear elements is researched with the help of the presented theorems. The restrictions on the initial data and parameters for the electrical circuit of the filter, which ensure the existence, the uniqueness and boundedness of global solutions, and the existence and uniqueness of  solutions with a finite escape time for the dynamics equation of the electrical circuit are obtained. The concrete functions and quantities (including nonlinear functions to be not global Lipschitz) defining the circuit parameters and satisfying the obtained restrictions are given. The numerical analysis of the mathematical model is carried out.

The paper has the following structure. The main theoretical results are given in Sections~\ref{SectThUst},~\ref{SectThNeust}, namely, in Section~\ref{SectThUst} (Section~\ref{SectThNeust}) the theorem on the Lagrange stability (instability) of the DAE is proved. In Sections~\ref{MatModUst},~\ref{MatModNeust} the mathematical model of the nonlinear radio engineering filter is researched with the help of the obtained theorems, conclusions and explanations of the obtained results from a physical point of view are presented (Subsections~\ref{ConclStab},~\ref{ConclInstab}) the presented, and the numerical analysis of the mathematical model is carried out (Subsections~\ref{NumStab},~\ref{NumInstab}). A more detailed discussion of the obtained results is given above. in Section~\ref{Prelim} we give a problem setting, preliminary information, definitions and corresponding explanations. In Section~\ref{Intro} the object and subject of the study are indicated, the actuality of the work is justified, the obtained results are discussed, auxiliary information, a literature review and a comparison with the known results are given. Section~\ref{Conclusions} contains general conclusions.

The following notation will be used in the paper: $E_X$\! is the identity operator in the space $X$; $\left. A\right|_X$ is the restriction of the operator $A$ to $X$;  $L(X,Y)$ is the space of continuous linear operators from $X$ to $Y$, $L(X,X)\!=\!L(X)$; $x^T$ is the transpose of $x$; the notation $\int\limits_c^{+\infty}\! f(t)\, dt\! <\! +\infty$ means that the integral converges; the notation $ \int\limits_c^{+\infty}\! f(t)\, dt \!=\! \infty$  means that the integral does not converge; $\delta _{ij}$ is the Kronecker delta. Sometimes the function $f$ is denoted by the same symbol $f(x)$ as its value at the point $x$ in order to emphasize (explicitly indicate) that the function depends on the variable $x$, but from the context it will be clear what exactly is meant.

 \section{Problem setting and preliminaries}\label{Prelim}

Consider the Cauchy problem for the semilinear DAE:
\[
 \frac{d}{dt}[Ax]+Bx=f(t,x), \eqno \eqref{DAE}
\]
\[
 x(t_0)=x_0,   \eqno \eqref{ini}
\]
\noindent where $t,\, t_0 \ge 0$, $x,\, x_0\!\in\!\Rn$, $f\colon [0,\infty)\times \Rn \to \Rn$  is a continuous function,  $A,\, B\colon \Rn\! \to \Rn$ are linear operators being corresponded to $n\times n$ matrices $A, B$. The operator $A$ is degenerate (noninvertible), the operator $B$ may also be degenerate. The matrix pencil (as well as the corresponding operator pencil) $\lambda A+B$ is \textit{regular}, i.e., $\det(\lambda A+B)\not \equiv 0$.
\begin{definition}
 A function $x(t)$ is called a \textit{solution} of the Cauchy problem \eqref{DAE}, \eqref{ini} on some interval $[t_0,t_1)$, $t_1\le \infty $, if $x \in C([t_0,t_1),\, \Rn)$, $Ax \in C^1([t_0,t_1),\Rn)$, $x$~satisfies the equation \eqref{DAE} on $[t_0,t_1)$ and the initial condition \eqref{ini}.
\end{definition}

It is assumed that $\lambda A+B$ is a regular \emph{pencil of index 1}, that is, there exist constants $C_1$,~$C_2 >0$ such that
\begin{equation}\label{index1}
   \left\|(\lambda A+B)^{-1}\right\| \le C_1,\quad  |\lambda|\ge C_2.
\end{equation}
For the pencil $\lambda A+B$ satisfying \eqref{index1} there exist the two pairs of mutually complementary projectors $P_1$, $P_2$ and $Q_1$, $Q_2$ (i.e.,  $P_i P_j \hm= \delta _{ij} P_i$, $P_1\!+\!P_2\!=\!E_{\Rn}$, and $Q_i Q_j \hm= \delta _{ij} Q_i$, $Q_1\!+\!Q_2\!=\!E_{\Rn}$, $i,j=1,2$) which were first introduced by A.G.~Rutkas \cite[Lemma~3.2]{Rut} and can be constructively determined by the formulas similar to \cite[(5),~(6)]{Rut_Vlas} (where $X=Y=\Rn$) or \cite[(3.4)]{Rut} (and for the real operators $A$, $B$ the projectors are real). These projectors decompose the space $\Rn$ into direct sums of subspaces
\begin{equation}\label{Proj.2}
 \Rn =X_1 \dot{+}X_2,\quad  \Rn=Y_1 \dot{+} Y_2,\quad  X_j =P_j \Rn,\quad Y_j =Q_j \Rn,\quad  j=1,2,
\end{equation}
such that the operators $A$, $B$ map $X_j$ into $Y_j$, and the induced operators $A_j \hm= \linebreak
=\left. A\right|_{X_j}\colon  X_j \to Y_j$, $B_j =\left. B\right|_{X_j}\colon X_j \to Y_j$, ${j=1,2}$, ($X_2=Ker\, A$,\, $Y_1=A\Rn =A_1 X_1$) are such that $A_2=0$, inverse operators $A_1^{-1}\! \in\! L(Y_1,X_1)$, ${B_2^{-1}\! \in\! L(Y_2,X_2)}$ exist (cf. \cite[Lemma~3.2]{Rut}, \cite[Sections~2,6]{Rut_Vlas}) and
\begin{equation}\label{Proj.3}
 AP_j =Q_j A,\quad BP_j =Q_j B,\quad j=1,2.
\end{equation}
With respect to the decomposition \eqref{Proj.2} any vector $x\in \Rn$ can be uniquely represented as the sum
\begin{equation}\label{xdecomp}
x=x_{p_1} +x_{p_2}, \quad x_{p_1} =P_1 x\in X_1,\; x_{p_2} =P_2 x\in X_2.
\end{equation}
This representation will be used in future. We will also use the auxiliary operator $G\in L(\Rn)$, (cf. \cite[Sections~2,~6]{Rut_Vlas})
\[
G=AP_1 +BP_2=A+BP_2,\quad GX_j =Y_j,\quad j=1,2,
\]
which has the inverse operator $G^{-1}\!\in\! L(\Rn)$ with the properties $G^{-1} AP_1 =P_1$, $G^{-1} BP_2 \hm=P_2$, $AG^{-1} Q_1 =Q_1$, $BG^{-1} Q_2 =Q_2$.

One can see that the projectors  $P_1$, $P_2$, $Q_1$, $Q_2$ allow to reduce the DAE to the equivalent system of a purely differential equation and a purely algebraic equation. Applying $Q_1$, $Q_2$ to the equation \eqref{DAE} and taking into account \eqref{Proj.3}, we obtain the equivalent system
\begin{equation*}
\left\{\begin{array}{l}
 \displaystyle{\frac{d}{dt}} (AP_1 x)+BP_1 x = Q_1 f(t,x), \\
 Q_2 f(t,x)-BP_2 x = 0.
 \end{array}\right.
\end{equation*}
\noindent Further, using $G^{-1}$, we obtain the system, which is equivalent to the DAE \eqref{DAE}:\,
\begin{equation}\label{prThreg1}
\left\{\begin{split}
&\frac{d}{dt} (P_1 x)=G^{-1} \big[ - BP_1 x + Q_1 f(t,P_1x+P_2x)\big], \\
& G^{-1} Q_2 f(t,P_1x+P_2x)-P_2 x  = 0.
\end{split}\right.
\end{equation}
\begin{remark}
We consider different notions of an index of the pencil, an index of a DAE, a relationship between them and their relationship with the mentioned notion of the pencil of index 1.  In \cite[Section~6.2]{Vlasenko1} the maximum length of the chain of an eigenvector and adjoint vectors of the matrix pencil $A+\mu B$ at the point $\mu = 0$ is called the index of the matrix pencil $\lambda A+B$.  Following \cite[Sections~6.2,~2.3.1]{Vlasenko1}, the regular pencil $\lambda A+B$ with the property \eqref{index1} is called a regular pencil of index 1.  Taking into account the properties of the projectors $P_j$, $Q_j$ and the induced operators $A_j$, $B_j$, ${j=1,2}$, if the condition \eqref{index1} holds, then the index of the pencil (or the index of nilpotency of the matrix pencil) $(A,B)$ is~1 in the sense as defined in the works of C.W. Gear, L.R. Petzold, for example, \cite[p.~717--718]{Gear_Petz84} (it is easy to verify using \cite[Thm.~2.2]{Gear_Petz84}).  In \cite[Def.~1.4]{Lamour-Marz-Tisch.}, the index of nilpotency of the matrix pencil $(A,B)$ \cite{Gear_Petz84} is called the Kronecker index of the regular matrix pair $\{A,B\}$ which forms the matrix pencil $\lambda A+B$. Also, by the index of a pencil one can determine the index of the corresponding system of differential-algebraic equations \cite[p.~718]{Gear_Petz84}, in particular, the index of the pencil $(A,B)$ (the Kronecker index of the regular matrix pair~$\{A,B\}$) coincides with the index of the linear DAE (the Kronecker index of the regular DAE \cite[Def.~1.4]{Lamour-Marz-Tisch.}) $A\frac{d}{dt}x+Bx=g(t)$. This is analogous to the fact that the pencil $\lambda A+B$ corresponds to the linear part $\frac{d}{dt}[Ax]+Bx$ of the DAE \eqref{DAE} and the influence of the linear part is determined by the properties of the corresponding pencil.
For comparison with the notion of the ``tractability index'' from the works of R.~M\"arz, C.~Tischendorf and R.~Lamour \cite{Marz,Tischendorf,Lamour-Marz-Tisch.}, note that the linear DAE ${\frac{d}{dt}[Ax]+Bx=q(t)}$ with the regular pencil $\lambda A+B$ of index 1 (i.e., \eqref{index1} is fulfilled) is regular with tractability index~1 \cite[p.~65,~91, Def.~2.25]{Lamour-Marz-Tisch.}.

But if we consider a semilinear DAE $A\frac{d}{dt}x+Bx=f(t,x)$ (in this case the solution $x(t)$ must be smoother than the solution of the DAE \eqref{DAE}), then in order for it to have the index~1 for all $t\ge 0$,~$x\in D\subseteq \Rn$ (to be exact, $(t,x)\in L_0$, where $L_0$ is defined in \eqref{soglreg2}), it is necessary that the pencil $\left(A, B-\frac{\partial }{\partial x}f(t,x)\right)$ has the index~1 for all $t\ge 0$,~$x\in D$. This condition is too restrictive for my research  %(a similar condition will be discussed in more detail below)
and, besides, it does not allow to prove the existence of a unique global solution, since the uniqueness of the solution can be proved only locally (in \cite[Ch.~9]{Ascher-Petz} this is shown for a ``semi-explicit index-1'' DAE).
\end{remark}

One of the conditions which allow to prove the existence of a unique global solution of the Cauchy problem \eqref{DAE}, \eqref{ini} for any consistent initial value $x_0$ is the condition of the basis invertibility of an operator function (Definition~\ref{bas-invert}) which will be discussed below. To begin, we introduce the definitions.
\begin{definition}\label{add-res}
A system of one-dimensional projectors $\{\Theta _k\}_{k=1}^s$,\, $\Theta _k \colon Z\to Z$, such that $\Theta _i\, \Theta _j \hm= \Theta _j\, \Theta _i = \delta _{ij}\, \Theta _i$ ($\delta _{ij}$ is the Kronecker delta), and $E_Z =\sum\limits_{k=1}^s\Theta _k$ is called an \emph{additive resolution of the identity} in $s$-dimensional linear space $Z$.
\end{definition}

The additive resolution of the identity generates a direct
decomposition of $Z$  into the sum of $s$ one-dimensional subspaces: $Z = Z_1 \dot + Z_2\dot +\cdots \dot +Z_s$,\; $Z_k =  \Theta _k\, Z$.
\begin{definition}\label{bas-invert}
 Let $W$, $Z$ be $s$-dimensional linear spaces, $D\subset W$.  An operator function (a mapping) $\Phi \colon D\hm\to L(W,Z)$ is called  \textit{basis invertible} on the convex hull  $conv\{\Hat w, \Hat{\Hat w}\}$ of vectors $\Hat w,\, \Hat{\Hat w} \in D$ if for any set of vectors $\{w^k\}_{k=1}^s$, $w^k\in conv \{\Hat w, \Hat{\Hat w}\}$,  and some additive resolution of the identity $\{\Theta _k\}_{k=1}^s$ in the space $Z$ the operator
\begin{equation*}
 \Lambda \hm=\sum\limits_{k=1}^s \Theta _k \Phi (w^k) \hm\in L(W,Z)
\end{equation*}
has the inverse operator $\Lambda ^{-1} \in L(Z,W)$.
\end{definition}

Let us represent the operator $\Phi (w)\!\in\! L(W,Z)$ as the matrix relative to some bases in the $s$-dimensional spaces $W$,~$Z$:\,
\begin{equation*}
\Phi (w)= \begin{pmatrix} \Phi _{\scriptscriptstyle 11} (w) & \cdots & \Phi _{{\scriptscriptstyle 1} s} (w) \\ \cdots & \cdots & \cdots \\ \Phi _{s {\scriptscriptstyle 1}} (w) & \cdots & \Phi _{ss} (w) \end{pmatrix}.
\end{equation*}
\noindent Definition \ref{bas-invert} can be stated as follows: the matrix function  $\Phi $ is called \textit{basis invertible} on the convex hull $conv \{\Hat w, \Hat{\Hat w}\}$ of vectors $\Hat w,\, \Hat{\Hat w}\! \in\! D$ if for any set of vectors $\{w^k\}_{k=1}^s \hm\subset conv \{\Hat w,\, \Hat{\Hat w}\}$ the matrix\,
\begin{equation*}
\Lambda = \begin{pmatrix} \Phi _{\scriptscriptstyle 11}(w^{\scriptscriptstyle 1}) & \cdots & \Phi _{{\scriptscriptstyle 1} s} (w^{\scriptscriptstyle 1}) \\ \cdots & \cdots & \cdots \\  \Phi _{s {\scriptscriptstyle 1}} (w^s) & \cdots & \Phi _{ss} (w^s) \end{pmatrix}
\end{equation*}
has the inverse $\Lambda ^{-1}$. %(i.e., $\det\Lambda \not = 0$).

Note that the property of basis invertibility does not depend on the choice of a basis or an additive resolution of the identity in $Z$ (this follows directly from the Definitions~\ref{add-res},~\ref{bas-invert}).

Obviously, if the operator function $\Phi $ is basis invertible on $conv \{\Hat w, \Hat{\Hat w}\}$, then it is invertible at any point  $w^* \in conv \{\Hat w, \Hat{\Hat w}\}$ ($w^*=\alpha \Hat{\Hat w} + (1-\alpha)\Hat w$, $\alpha \in [0,1]$),
i.e., for each point $w^*\in conv \{\Hat w, \Hat{\Hat w}\}$ its image $\Phi (w^*)$ under the mapping $\Phi $ is an invertible continuous linear operator from $W$ to $Z$. The converse is not true unless the spaces $W$, $Z$ are one-dimensional. We give an example.
\begin{example}
 Let $W = Z = {\mathbb R}^2$, $D = conv \{\Hat w, \Hat{\Hat w}\}$, $\Hat w =(1,-1)^T$, $\Hat{\Hat w}=(1,1)^T$, $w =(a, b)^T \in D$,
\begin{equation*}
 \Phi (w) = \begin{pmatrix}  a\, b & 1 \\ -1 & a\, b \end{pmatrix}.
\end{equation*}
For the set of vectors $\{w^1, w^2\} \subset conv \{\Hat w ,\Hat{\Hat w}\}$, $w^1= (a_1, b_1)^T$, $w^2\hm= (a_2, b_2)^T$, the~operator $\Lambda$ has the form
\begin{equation*}
 \Lambda = \begin{pmatrix} a_1\, b_1 & 1 \\ -1 & a_2\, b_2 \end{pmatrix}.
\end{equation*}
Since $\det \Phi (w) \hm= a^2\, b^2+1 \ne 0$ for any $w \in D$, then $\Phi (w)$ is invertible on $D$. However the operator $\Lambda$ is not invertible for $\{w^1, w^2\} = \{\Hat w, \Hat{\Hat w}\}$ and hence the operator function $\Phi $  is not basis invertible on $D$. If we take $\Hat w =(1,0)^T$, then $\Phi $  will be basis invertible on~$D$.
\end{example}

Now we discuss why this definition is needed. As shown above, the DAE \eqref{DAE} is equivalent to the system of a purely differential equation and a purely algebraic equation. The algebraic equation defines one of the components of a DAE solution as an implicitly given function. With the help of the implicit function theorem this component can be defined as an (unique) explicitly given function, but only locally, i.e., in some sufficiently small neighborhood. But we need a unique globally defined explicit function for further application of the results on Lagrange stability to the differential equation, which will be obtained by substitution of the found component (function). For this purpose, the condition of \emph{the basis invertibility of an operator function} (Definition~\ref{bas-invert}), which was first introduced in \cite{RF1}, is used.  Note that this condition does not impose restrictions of a type of a global Lipschitz condition, including the condition of contractivity, and does not require the global boundedness of the norm for an inverse function on the whole domain of definition (see Remark~\ref{RemFuncPhi}).

In the theorem \cite[Thm.~6.7]{Lamour-Marz-Tisch.} conditions of the global solvability are given for the nonlinear DAE $f\big((D(t)x)',x,t\big)=0$\, \cite[(4.1)]{Lamour-Marz-Tisch.} that is a contractive regular index-1 DAE for all $Dx\in {\mathbb R}^n$, $x\in {\mathbb R}^m$, $t\in [0,\infty )$. The global condition of index~1 for the DAE that is present in this theorem means that for all $y\in {\mathbb R}^n$, $x\in {\mathbb R}^m$, $t\in [0,\infty )$ the pencil $\lambda f_y(y,x,t)D(t) + f_x(y,x,t)$ is regular with Kronecker index~1 \cite[p.~318-320]{Lamour-Marz-Tisch.} . Therefore, there must exist the constants $C_1,\, C_2 >0$ independent of $t$, $x$, $y$ and such that $\left\|(\lambda f_y(y,x,t)D(t) + f_x(y,x,t))^{-1}\right\| \le C_1$ for all $y\in {\mathbb R}^n$, $x\in {\mathbb R}^m$, $t\in [0,\infty )$, $|\lambda|\ge C_2$, i.e., the norm is globally bounded.
The requirement of the contractivity of the regular index-1 DAE (see \cite[Def.~6.1,~6.5]{Lamour-Marz-Tisch.}) imposes additional restrictions.
Taking into account Remark~\ref{RemFuncPhi}, in the case of a semilinear DAE \emph{these requirements are more restrictive than the conditions of global solvability from Theorem~\ref{Th_Ust1}}.

Concerning the theorems \cite[Thm.~2.1]{Marz}, \cite[Thm.~3.3]{Tischendorf}, note that they are obtained for the autonomous DAE. If we consider the nonautonomous DAEs, namely, $A\frac{d}{dt}x\hm+g(t,x)=0$ or $f(x',x,t)=0$, where $f(x',x,t)= Ax'+g(t,x)$, then, as said above, the requirement that the pencil $\lambda A + \frac{\partial }{\partial x}g(t,x^*)$ has index~1 means that there exist the constants $C_1, C_2\! >\!0$ independent of $t$ and such that $\left\|(\lambda A+\frac{\partial }{\partial x}g(t,x^*))^{-1}\right\|\! \le\! C_1$, $|\lambda|\!\ge\! C_2$ for all $t\!\in\! [0,\infty)$, i.e., the norm is globally bounded in $t$, hence, this requirement is more restrictive than the requirement that the operator function $\Phi$ is basis invertible and the pencil $\lambda A +B$ has index~1.

Also, note that in \cite[Thm.~2.1]{Marz} and \cite[Thm.~3.3]{Tischendorf} it is required that the nonlinear function is twice continuously differentiable, while in Theorem~\ref{Th_Ust1} it is only required that $f(t,x)$ is continuous and has the continuous partial derivative $\frac{\partial }{\partial x} f(t,x)$.

\vspace*{0.3cm}

 A solution $x(t)$ of the Cauchy problem \eqref{DAE}, \eqref{ini} is called \emph{global} if it exists on the whole interval $[t_0,\infty )$.
\begin{definition}\label{Def-FiniteTime}
 A solution $x(t)$ of the Cauchy problem \eqref{DAE}, \eqref{ini} has a  \emph{finite escape time} if it exists on some finite interval $[t_0,T)$ and is unbounded, i.e., there exists $T\!<\!\infty$ $(T\!>\!t_0)$ such that $\mathop{\lim }\limits_{t\to T-0} \| x(t)\| =+\infty $.\,
\end{definition}
If the solution has a finite escape time, it is called \emph{Lagrange unstable}.
\begin{definition}\label{Def-LStableSol}
  A solution $x(t)$ of the Cauchy problem \eqref{DAE}, \eqref{ini} is called \emph{Lagrange stable} if it is global and bounded, i.e., the solution $x(t)$ exists on $[t_0,\infty )$ and $\mathop{\sup }\limits_{t\in [t_0,\infty )}\! \| x(t)\|\!\hm<\! +\infty$.
\end{definition}
\begin{definition}\label{Def-EqLStUnst}
\emph{The equation \eqref{DAE} is Lagrange stable} if every solution of the Cauchy problem \eqref{DAE}, \eqref{ini} is Lagrange stable.

\emph{The equation \eqref{DAE} is Lagrange unstable} if every solution of the Cauchy problem \eqref{DAE}, \eqref{ini} is Lagrange unstable.
\end{definition}

 \section{Lagrange stability of the semilinear DAE}\label{SectThUst}

The theorem on the Lagrange stability of the DAE \eqref{DAE}, which gives sufficient conditions of the existence and uniqueness of global solutions of the Cauchy problem \eqref{DAE}, \eqref{ini}, where the initial points satisfy the consistency condition $(t_0,x_0)\in L_0$ (the manifold $L_0$ is defined in \eqref{soglreg2}), and gives conditions of the boundedness of the global solutions, is presented below.

\begin{theorem}\label{Th_Ust1}
Let $f\in C([0,\infty)\times \Rn, \Rn)$ have the continuous partial derivative $\frac{\partial }{\partial x} f(t,x)$ on $[0,\infty)\times \Rn$,\, $\lambda A+B$ be a regular pencil of index 1 and
\begin{equation}\label{soglreg2}
\forall\, t\!\ge\! 0\, \forall\, x_{p_1}\! \in\! X_1\, \exists\, x_{p_2}\! \in\! X_2 \colon\! (t,x_{p_1}\!+x_{p_2})\!\in\! L_0\! =\! \{(t,x)\! \in\!  [0,\infty) \times  \Rn  \mid  Q_2[Bx-f(t,x)]\! =\! 0\},\!\!
\end{equation}
where $X_1$, $X_2$ from \eqref{Proj.2}. Let for any %fixed
$\Hat x_{p_2}, \Hat{\Hat x}_{p_2}\! \hm\in\! X_2$ such that $(t_*, x_{p_1}^*\! + \Hat{x}_{p_2})$,$(t_*, x_{p_1}^*\! + \Hat{\Hat x}_{p_2})\!\hm\in\! L_0$ the operator function %defined as
\begin{equation}\label{funcPhi}
\Phi \colon X_2 \to L(X_2,Y_2),\quad \Phi (x_{p_2})=\left[\frac{\partial }{\partial x} \big(Q_2 f(t_*,x_{p_1}^* +x_{p_2})\big)-B\right] P_2,
\end{equation}%(\,$t_*$, $x_{p_1}^*$ are fixed\,)
be basis invertible on the convex hull $conv\{\Hat x_{p_2},\, \Hat{\Hat x}_{p_2}\}$. Suppose for some self-adjoint positive operator $H \in L(X_1)$ and some number $R>0$ there exist functions $k\in C([0,\infty),{\mathbb R})$, $U\hm\in C((0,\infty), (0,\infty))$ such that
\begin{equation*}
\int\limits_c^{+\infty} \frac{dv}{U(v)} =+\infty\quad (c>0),
\end{equation*}

\vspace*{-0.4cm}

\begin{equation}\label{Lagr1}
\hspace*{-0.15cm}\!\!\! \left(HP_1 x,G^{-1}[-BP_1 x +Q_1 f(t,x)]\right)\!\le k(t)\, U\!\left({\textstyle \frac{1}{2}}(HP_1x,P_1x)\right)\;\, \forall\, (t,x)\!\in\! L_0\!:\!\| P_1 x\|\!\ge\!R.\!\!
\end{equation}
Then for each initial point $(t_0,x_0)\in L_0$ there exists a unique solution $x(t)$ of the Cauchy problem \eqref{DAE}, \eqref{ini} on $[t_0,\infty)$.

If, additionally,

\noindent{\centering $\int\limits_{t_0}^{+\infty} k(t)\, dt <+\infty$,

\vspace*{0.1cm}

}
\noindent there exists $\Tilde{x}_{p_2}\! \in\! X_2$ such that for any $\Tilde{\Tilde x}_{p_2}\!\in X_2$ such that $(t_*, x_{p_1}^*\! +\! \Tilde{\Tilde x}_{p_2})\!\in\! L_0$ the operator function \eqref{funcPhi} is basis invertible on $conv\{\Tilde x_{p_2}, \Tilde{\Tilde x}_{p_2}\}\setminus \{\Tilde{x}_{p_2}\}$, and
\begin{equation}\label{LagrA1}
\mathop{\sup }\limits_{t\in [0,\infty ),\: \| x_{p_1}\|\le M} \|Q_2 f(t,x_{p_1}+ \tilde{x}_{p_2})\| < +\infty,\qquad M<\infty \text{ is a constant,}
\end{equation}
then for the initial points $(t_0,x_0)\in L_0$ the equation \eqref{DAE} is Lagrange stable.
\end{theorem}
\begin{remark}\label{RemFuncPhi}
Now we explain the restriction which is imposed on $\Phi $ \eqref{funcPhi} (for the existence and uniqueness of global solutions). In the case when the space $X_2$ is one-dimensional (then the basis invertibility is equivalent to the invertibility), it is required that the continuous linear operator $\Lambda=\Phi (x^*_{p_2})$, $x^*_{p_2}\!\in\! conv\{\Hat x_{p_2},\, \Hat{\Hat x}_{p_2}\}$, has a continuous linear  inverse operator for any fixed $\Hat x_{p_2}$, $\Hat{\Hat x}_{p_2}$, $t^*$, $x^*_{p_1}$ such that $(t^*, x^*_{p_1}\! + \Hat x_{p_2}), (t^*, x^*_{p_1}\! + \Hat{\Hat x}_{p_2})\!\!\hm\in\! L_0$. In the case when the dimension of $X_2$ greater than 1, it is required that the operator $\Lambda \in L(X_2,Y_2)$, which is constructed from the operator function $\Phi$ (as shown in Definition~\ref{bas-invert}) for fixed $\Hat x_{p_2}$, $\Hat{\Hat x}_{p_2}$, $t^*$, $x^*_{p_1}$ such that $(t^*, x^*_{p_1}\! + \Hat x_{p_2}), (t^*, x^*_{p_1}\! \hm+ \Hat{\Hat x}_{p_2})\!\hm\in\! L_0$, is invertible.   At the same time, it does not require the global boundedness of the norm of the mapping $\left[\Phi \right]^{-1}$ on $X_2$ and does not require the global boundedness of the norm of the function $\left[\frac{\partial }{\partial x} \big(Q_2 f(t,x_{p_1} +x_{p_2})\big)P_2-BP_2\right]^{-1}$ on $[0,\infty)\times \Rn$ (i.e, it is not required that the norm of the function is bounded by a constant for all $t$,~$x_{p_1}$,~$x_{p_2}$). For comparison, the condition of index~1 for the DAE was discussed above.
\end{remark}
\begin{proof}
The DAE \eqref{DAE} is equivalent to the system \eqref{prThreg1} (as shown in Section~\ref{Prelim}). Denote $\dim X_1 =a$, $\dim X_2 =d$ $(d=n-a)$. Any  vector $x\in \Rn$ can be represented as $x =\left(\begin{array}{c} z \\ u \end{array}\right) \hm\in {\mathbb R}^a \times {\mathbb R}^d$, where $z\in {\mathbb R}^a$, $u \in {\mathbb R}^d$ are column vectors. We introduce the operators (the method of the construction of the operators is presented in \cite[Section~2]{RF1}) $P_a\colon {\mathbb R}^a \to X_1$,
$P_d\colon {\mathbb R}^d \to X_2$, which have the inverse operators  $P_a^{-1}\colon X_1 \to {\mathbb R}^a$, $P_d^{-1}\colon X_2 \to {\mathbb R}^d$. Then $z=P_a^{-1} P_1 x$, $u=P_d^{-1} P_2 x$, $x=P_a z+P_d u$ (recall that~\eqref{xdecomp}), and $P_a^{-1} P_1 P_a =E_{{\mathbb R}^a}$, $P_d^{-1} P_2 P_d =E_{{\mathbb R}^d}$. Multiplying the equations of the system~\eqref{prThreg1} by $P_a^{-1} $, $P_d^{-1}$ correspondingly and making the change $P_1 x=P_a z$, $P_2 x =P_d u$, we get the equivalent system
\begin{eqnarray}
& &\frac{d}{dt} z=P_a^{-1} G^{-1}\big[- B P_a z + Q_1 \tilde{f}(t,z,u)\big], \label{prThreg2} \\
& &P_d^{-1} G^{-1} Q_2 \tilde{f}(t,z,u)-u = 0,  \label{prThreg3}
\end{eqnarray}

\vspace*{-0.1cm}

\noindent where $\tilde{f}(t,z,u)=f(t,P_a z+P_d u)$.

Thus, the semilinear DAE \eqref{DAE} is equivalent to the system \eqref{prThreg2},~\eqref{prThreg3}.

\emph{I (Existence and uniqueness).} First we prove the first part of the theorem, that is, the existence and uniqueness of global solutions.

Consider the mapping

\vspace*{-0.2cm}

\begin{equation}\label{prThregF}
F(t,z,u)=P_d^{-1} G^{-1} Q_2 \tilde{f}(t,z,u)-u.
\end{equation}

\vspace*{-0.1cm}

\noindent It is continuous on $[0,\infty)\times {\mathbb R}^a \times {\mathbb R}^d$ and has continuous partial derivatives

\vspace*{-0.6cm}

\begin{equation*}
\begin{split}
\frac{\partial }{\partial z} F(t,z,u)&=P_d^{-1} G^{-1} \frac{\partial }{\partial x} (Q_2 f(t,x))P_a, \\
\frac{\partial }{\partial u} F(t,z,u)&=P_d^{-1} \left[G^{-1} \frac{\partial }{\partial x} (Q_2 f(t,x))-P_2 \right]P_d =P_d^{-1} G^{-1} \Phi (P_d u)P_d,
\end{split}
\end{equation*}

\vspace*{-0.2cm}

\noindent where $\Phi $ is the operator function \eqref{funcPhi}, $\Phi (P_d u)=\Phi (x_{p_2})$, $x_{p_2}=P_d u\in X_2$.

Let us prove that for any $\Hat u,\, \Hat{\Hat u}\! \in\! {\mathbb R}^d$ such that $(t_*,z_*,\Hat u)$, $(t_*,z_*,\Hat{\Hat u})\! \in\! \tilde{L}_0$, where

\vspace*{-0.1cm}

\begin{equation}\label{prThregL0}
  \tilde{L}_0 =\left\{(t,z,u)\in [0,\infty)\times {\mathbb R}^a \times {\mathbb R}^d\; \big|\; P_d^{-1} G^{-1} Q_2 \tilde{f} (t,z,u)-u=0\right\},
\end{equation}

\vspace*{-0.1cm}

\noindent the operator function $\Psi \colon {\mathbb R}^d \to L({\mathbb R}^d)$,  $\Psi (u)\hm=\frac{\partial }{\partial u} F(t_*,z_*,u)$, is basis invertible on \linebreak
$conv\{\Hat u, \Hat{\Hat u}\}$. Since \eqref{funcPhi} is basis invertible on $conv\{\Hat x_{p_2}, \Hat{\Hat x}_{p_2}\}$ for any ${\Hat x_{p_2},\, \Hat{\Hat x}_{p_2} \in X_2}$  such that $(t_*, x_{p_1}^*\! + \Hat{x}_{p_2})$, $(t_*, x_{p_1}^*\! + \Hat{\Hat x}_{p_2})\!\hm\in\! L_0$, there exists an additive resolution of the identity $\{\Theta _k \}_{k=1}^d $ in $Y_2$ such that the operator $\Lambda _1 =\sum\limits_{k=1}^d \Theta _k \Phi (x_{p_2}^k) \in L(X_2,Y_2)$  is invertible for any set of vectors $\{x_{p_2}^k\}_{k=1}^d \subset conv\{\Hat x_{p_2}, \Hat{\Hat x}_{p_2}\}$. With the help of the invertible operator $N=P_d^{-1} G^{-1} :Y_2 \to {\mathbb R}^d$ we introduce the system of one-dimensional projectors $\hat{\Theta }_k =N\Theta _k N^{-1} $, which form the additive resolution of the identity $\{\hat{\Theta }_k \}_{k=1}^d$ in ${\mathbb R}^d$. Take any $\Hat u,\, \Hat{\Hat u}\! \in\! {\mathbb R}^d$ such that $(t_*,z_*,\Hat{u})$, $(t_*,z_*,\Hat{\Hat u})\in \tilde{L}_0$
and any $u^k\! \in\! conv\{\Hat u, \Hat{\Hat u}\}$, $k=\overline{1,d}$. Taking into account that $(t,z,u)\!\in\! \tilde{L}_0 \hm\Leftrightarrow (t,x_{p_1}+x_{p_2})\hm\!\in\! L_0$ and for $\Hat x_{p_2}=P_d \Hat u$, $\Hat{\Hat x}_{p_2} =P_d \Hat{\Hat u}$, $x_{p_2}^k =P_d u^k$, $x_{p_1}^*=P_a z_*$  the operator $\Lambda_1$ is invertible, the operator

\vspace*{-0.2cm}

\[
\Lambda_2 =\sum\limits_{k=1}^d\hat{\Theta }_k \frac{\partial }{\partial u} F(t_*,z_*,u^k) =\sum\limits_{k=1}^d\hat{\Theta }_k P_d^{-1} G^{-1} \Phi (P_d u^k) P_d  =N\Lambda _1 P_d
\]

\vspace*{-0.1cm}

\noindent acting in ${\mathbb R}^d$ is also invertible. Hence, $\Psi $ is basis invertible on $conv\{\Hat u,\! \Hat{\Hat u}\}$\!.

Let $(t_*, z_*)$ be an arbitrary (fixed) point of $[0,\infty)\times {\mathbb R}^a$. Choose $u_* \!\in\!  {\mathbb R}^d$ so that $(t_*,z_*,u_*)\!\hm\in\! \tilde{L}_0$, this is possible by virtue of the condition \eqref{soglreg2}. From the basis invertibility of $\Psi $, it follows that there exists the continuous linear inverse operator \linebreak
$\left[\frac{\partial }{\partial u} F(t_*,z_*,u_*)\right]^{-1}$. By the implicit function theorems  \cite{Schwartz1}, there exist neighborhoods $U_\delta (t_*,z_*)\! \hm=\! U_{\delta _1}(t_*)\times U_{\delta _2 }(z_*)$ (if $t_*\!=0$, then $U_{\delta _1 }\!(t_*)\!=[0,{\delta _1})$), $U_{\varepsilon }(u_*)$ and a unique function $u =u(t,z)\!\hm\in\! C(U_\delta (t_*,z_*), U_{\varepsilon } (u_*))$, being continuously differentiable in $z$, such that $F(t,z,u(t,z))=0$, $(t,z)\!\in\! U_\delta (t_*,z_*)$, and $u(t_*,z_*)=u_*$. We define a global function $u = \eta (t,z)\colon [0,\infty )\times {\mathbb R}^a \to {\mathbb R}^d$ at the point $(t_*,z_*)$ as $\eta (t_*,z_*)= u(t_*,z_*)$.

Let us prove that
\begin{equation}\label{prThreg4}
\forall (t,z)\in [0,\infty )\times {\mathbb R}^a\, \exists !\, u\in {\mathbb R}^d : (t,z,u)\in \tilde{L}_0.
\end{equation}
Consider arbitrary (fixed) points $(t_*,z_*,\Hat u)$, $(t_*,z_*,\Hat{\Hat u})\!\in\! \tilde{L}_0$, clearly, $F(t_*,z_*,\Hat u)=0$, \linebreak $F(t_*,z_*,\Hat{\Hat u})\hm=0$. The projections $F_k(t,z,u)\hm=\hat{\Theta }_k F(t,z,u)$, $k=\overline{1,d}$, are functions with values in the one-dimensional spaces $R_k =\hat{\Theta }_k {\mathbb R}^d$ being isomorphic to ${\mathbb R}$. According to the formula of finite increments \cite{Schwartz1}, $F_k (t_*,z_*,\Hat{\Hat u})\hm-F_k(t_*,z_*,\Hat u) \hm= \frac{\partial }{\partial u} F_k(t_*,z_*,u^k)(\Hat{\Hat u} -\Hat u)=0$, $u^k \in conv\{\Hat u, \Hat{\Hat u}\}$, $k = \overline{1,d}$. Hence $\hat{\Theta }_k \frac{\partial }{\partial u} F(t_*,z_*,u^k)(\Hat{\Hat u} -\Hat u) = 0$, $k = \overline{1,d}$, from which, summing these expressions over $k$, we obtain that $\Lambda _2 (\Hat{\Hat u} -\Hat u)=0$, where the operator $\Lambda_2 \hm=\sum\limits_{k=1}^d\hat{\Theta }_k \frac{\partial }{\partial u} F(t_*,z_*,u^k)\hm=\sum\limits_{k=1}^d\hat{\Theta }_k \Psi (u^k)$ is invertible by virtue of the basis invertibility of $\Psi $ (see above). Consequently,  $\Hat{\Hat u} = \Hat u$.

It is proved that \eqref{prThreg4} and in some neighborhood of each point  $(t_*,z_*)\in [0,\infty )\times {\mathbb R}^a$ there exists a unique solution $u=\nu (t,z)$ of \eqref{prThreg3}, which is continuous in $(t, z)$ and is continuously differentiable in $z$.  So the function $u=\eta (t,z)$ coincides with $\nu (t,z)$ in this neighborhood and it is a solution of \eqref{prThreg3} with the corresponding smoothness properties. Let us show that the function $u=\eta (t,z)$ is unique on the whole domain of definition. Indeed, if there exists a function $u=\mu (t,z)$ having the same properties as $u=\eta (t,z)$ at some point $(t_*,z_*)\in [0,\infty )\times {\mathbb R}^a$, then by \eqref{prThreg4} $\eta (t_*,z_*)=\mu (t_*,z_*)=u_*$. Therefore,  $\eta (t,z)=\mu (t,z)$ on $[0,\infty )\times {\mathbb R}^a$.

Substituting the function $u=\eta (t,z)$ in \eqref{prThreg2} and denoting $g(t,z)= Q_1 \tilde{f}(t,z,\eta (t,z))$, we get:
\begin{equation}\label{prUstreg2}
\frac{d}{dt} z=P_a^{-1} G^{-1}[- BP_a z + g(t,z)].
\end{equation}

By the properties of $\eta $, $Q_1 \tilde{f}$, the function $g(t,z)$  is continuous in $(t,z)$ and is continuously differentiable in $z$ on $[0,\infty )\times {\mathbb R}^a$. Hence, for each initial point $(t_0, z_0)$ such that $(t_0,z_0,\eta (t_0,z_0))\in \tilde{L}_0$ there exists a unique solution $z(t)$ of the Cauchy problem for the equation \eqref{prUstreg2} on some interval $[t_0,\varepsilon )$ with the initial condition $z(t_0)=z_0$. Note that if $(t_0,x_0)\in L_0$ and $x_0 =P_a z_0 +P_d \eta (t_0 ,z_0)$, then $(t_0, z_0, \eta (t_0, z_0))\in \tilde{L}_0$.

Introduce the function $V(P_1 x)=\frac{1}{2}(HP_1 x ,P_1 x)=\frac{1}{2} (HP_a z,P_a z)= \frac{1}{2} \left(P_a^* HP_a z,z\right) %= \linebreak
\hm=\frac{1}{2} (\hat{H}z,z)=\hat{V}(z)$, where $\hat{H}=P_a^* HP_a$ and $H$ is an operator from \eqref{Lagr1}. Then $\grad \hat{V}(z)\hm=\hat{H}z$, where $\grad \,\hat{V}$ is the gradient of the function $\hat{V}$. Since $\big( HP_a z, \linebreak
G^{-1}[-BP_a z + Q_1 f(t,P_a z \hm+ P_d \eta (t,z))]\big) \hm= \big(\hat{H}z,P_a^{-1} G^{-1} [-BP_a z \hm+ g(t,z)]\big)$, then according to \eqref{Lagr1} there exists $\hat{R}>0$  such that
\begin{equation}\label{L1.1}
 (\hat{H}z,P_a^{-1} G^{-1}[-BP_a z + g(t,z)])\le k(t)\, U(\hat{V})\quad \forall\, t \ge 0,\: \| z\| \ge \hat{R},
\end{equation}
where $k\in C([0,\infty),{\mathbb R})$, and $U \in C((0,\infty), (0,\infty))$ such that

\vspace*{0.1cm}

\noindent{\centering ${\int\limits_c^{+\infty}\frac{dv}{U(v)}=+\infty}$.

\vspace*{0.1cm}

}
Taking into account \eqref{L1.1}, for all $t\ge 0$ and all $z$ such that $\| z\|\ge \hat{R}$ the derivative $\left. \dot{\hat{V}}\right|_{\eqref{prUstreg2}}$ of the function $\hat{V}$ along the trajectories of \eqref{prUstreg2} (see the definition in \cite[Ch.~2]{LaSal-Lef-eng}) satisfies the following estimate:
\[
\left. \dot{\hat{V}}\right|_{\eqref{prUstreg2}} =(\hat{H}z,P_a^{-1} G^{-1}[-BP_a z + g(t,z)])\le  k(t)\, U(\hat{V}).
\]
It follows from the properties of the functions $k$, $U$ that the inequality $\dot{v}\le k(t)U(v)$, $t\ge 0$, has no positive solution $v(t)$ with finite escape time (see \cite[Ch.~4]{LaSal-Lef-eng}). Then by~\cite[Ch.~4, Theorem~XIII]{LaSal-Lef-eng} every solution $z(t)$ of \eqref{prUstreg2} is defined in the future (i.e., the solution is defined on $[t_0,\infty )$). Consequently, the function $x(t) =P_a z(t)+P_d \eta (t,z(t))$ is a solution of the Cauchy problem \eqref{DAE}, \eqref{ini} on $[t_0,\infty)$.

Let us verify the \emph{uniqueness} of the global solution. It follows from what has been proved that the global solution $x(t)$ is unique on some interval $[t_0,\varepsilon )$. Assume that the solution is not unique on $[t_0,\infty )$. Then there exist $t_* \ge \varepsilon $ and two different global solutions $x(t)$, $\tilde{x}(t)$ with the common value $x_*\!=x(t_*)\!=\tilde{x}(t_*)$.
Let us take the point $(t_* ,x_*)$ as an initial point, then there must be a unique solution of \eqref{DAE} on some interval $[t_*,\varepsilon _1)$ with the initial value $x(t_*)=x_*$, which contradicts the assumption.

\emph{II (Boundedness).} Prove the second part of the theorem, that is, the Lagrange stability of the DAE. Suppose that the additional conditions of the theorem are satisfied.

Since $\int\limits_{t_0}^{+\infty} k(t)\, dt <+\infty$, the inequality $\dot{v}\le k(t)U(v)$, $t\ge 0$, has no unbounded positive solution for $t\ge 0$ (see \cite[Ch.~4]{LaSal-Lef-eng}). Then by \cite[Ch.~4, Theorem~XV]{LaSal-Lef-eng} the equation \eqref{prUstreg2} is Lagrange stable. Hence, $\mathop{\sup }\limits_{t\in [t_0,\infty )} \| z(t)\|<+\infty$, i.e.,
\begin{equation}\label{L1.2}
 \exists\, M_* <\infty :\|z(t)\|\le M_*\; \forall\, t \in [t_0,\infty).
\end{equation}

Taking into account the properties of $\Phi $ \eqref{funcPhi} and the connection between $\Phi $ and the operator function $\Psi \colon {\mathbb R}^d \to L({\mathbb R}^d)$ introduced in item I of the proof, we get that there exists a point $\tilde{u} \in {\mathbb R}^d$ ($\tilde{u} = P_d^{-1} \tilde{x}_{p_2}$) such that for any $\Tilde{\Tilde u} \in {\mathbb R}^d$ such that $(t_*, z_*,\Tilde{\Tilde u})\! \in\! \tilde{L}_0$ the operator function $\Psi$ is basis invertible on $conv\{\Tilde{u},\, \Tilde{\Tilde u}\}\setminus \{\tilde{u}\}$.  Let $(t_*, z_*,\Tilde{\Tilde u})\! \in\! \tilde{L}_0$ be an arbitrary (fixed) point and $\tilde{u}$ be a point with the property indicated above. Then using the formula of finite increments for $F_k(t_*,z_*,\Tilde{\Tilde u})$ and $F_k(t_*,z_*,\tilde{u})$, where $F_k(t,z,u)=\hat{\Theta }_k F(t,z,u)$, $F$ is the mapping \eqref{prThregF} and $\{\hat{\Theta }_k\}_{k=1}^d$ is an additive resolution of the identity in ${\mathbb R}^d$, and summing the obtained equalities over $k$ we get that $F(t_*,z_*,\Tilde{\Tilde u})- F(t_*,z_*,\tilde{u})= \Lambda _2 (\Tilde{\Tilde u}-\tilde{u})$, where $\Lambda_2 =\sum\limits_{k=1}^d\hat{\Theta }_k \Psi (u^k)$,\, $\Psi (u^k)=  \frac{\partial }{\partial u} F(t_*,z_*,u^k)$, $u^k \in conv\{\tilde{u},\, \Tilde{\Tilde u}\}\setminus \{\tilde{u}\}$, i.e., $u^k=\alpha \Tilde{\Tilde u} + (1-\alpha) \tilde{u}$, $\alpha \in (0,1]$, $k = \overline{1,d}$.
It follows from the basis invertibility of $\Psi$ on $conv\{\tilde{u},\, \Tilde{\Tilde u}\}\setminus \{\tilde{u}\}$ that there exists the inverse operator $\Lambda_2^{-1} \in L({\mathbb R}^d)$.  Taking into consideration the above facts and the fact that $F(t_*,z_*,\Tilde{\Tilde u})=0$ we get $\Tilde{\Tilde u} = \tilde{u} \hm- \Lambda_2^{-1}[P_d^{-1} G^{-1} Q_2 \tilde{f}(t_*,z_*,\tilde{u}) - \tilde{u}]$. This is fulfilled for an arbitrary point $(t_*, z_*,\Tilde{\Tilde u})\! \in\! \tilde{L}_0$. Consequently, for each $t_*\! \in\! [t_0,\infty)$ the equality $\eta (t_*,z(t_*))\!= \tilde{u} \hm- \Lambda_2^{-1}[P_d^{-1} G^{-1} Q_2 \tilde{f}(t_*,z(t_*),\tilde{u}) \hm- \tilde{u}]$, where $z(t)$ and $\eta (t,z(t))$ are components of the global solution $x(t) =P_a z(t)\hm+P_d \eta (t,z(t))$ of the Cauchy problem \eqref{DAE}, \eqref{ini},  holds.  Denote $\tilde{M}= \|\tilde{u}\|$. Taking into account that $\Lambda_2^{-1}$ is a bounded linear operator (since $\Lambda_2^{-1}\in L({\mathbb R}^d)$), there exists a constant  ${N\ge0}$ such that $\|\eta (t_*,z(t_*))\| \hm\le (1+N)\tilde{M} \hm+ N \|P_d^{-1}G^{-1}\|\, \|Q_2 \tilde{f}(t_*,z(t_*),\tilde{u})\|$ for each $t_* \in [t_0,\infty)$. Then it follows from \eqref{L1.2}, \eqref{LagrA1} that there exists a constant $C\ge 0$ such that $\|\eta (t_*,z(t_*))\| \hm\le C$ for each $t_* \in [t_0,\infty)$.

Since the estimate $\|x(t)\| = \|P_a z(t) \hm+ P_d \eta (t,z(t))\|\le \|P_a \|M_*+ \|P_d \| C$ is fulfilled for all $t \in [t_0,\infty)$, the solution $x(t)$ of the Cauchy problem \eqref{DAE}, \eqref{ini} is Lagrange stable. This holds for each initial point $(t_0,x_0)\in L_0$. Hence for the initial points $(t_0,x_0)\in L_0$ the equation \eqref{DAE} is Lagrange stable.

The theorem is proven.
\end{proof}
\begin{remark}\label{RemConsistIni}\!\!
 \emph{The consistency condition} $(t_0,x_0)\!\in\! L_0$ for the initial point $(t_0,x_0)$ is one of the necessary conditions for the existence of a solution of the Cauchy problem~\eqref{DAE},~\eqref{ini}.
\end{remark}
\begin{remark}%Note that
If $\Phi$ \eqref{funcPhi} is basis invertible on $conv\{\Hat x_{p_2}, \Hat{\Hat x}_{p_2}\}$ for any $\Hat x_{p_2},\, \Hat{\Hat x}_{p_2}\!\hm\in\! X_2$, $t_*\!\in\! [0,\infty)$, $x^*_{p_1}\!\in\! X_1$,
then obviously it is basis invertible on $conv\{\Hat x_{p_2}, \Hat{\Hat x}_{p_2}\}$ for any $\Hat{x}_{p_2},\, \Hat{\Hat x}_{p_2}$ such that $(t_*, x_{p_1}^*\! + \Hat{x}_{p_2}),\, (t_*, x_{p_1}^*\! + \Hat{\Hat x}_{p_2})\!\hm\in\! L_0$ and on $conv\{\tilde{x}_{p_2}, \Tilde{\Tilde x}_{p_2}\}\setminus \{\tilde{x}_{p_2}\}$ for any $\tilde{x}_{p_2}$ and any $\Tilde{\Tilde x}_{p_2}$ such that $(t_*, x_{p_1}^* + \Tilde{\Tilde x}_{p_2})\!\in\! L_0$.  It is clear that this requirement is stronger, however, its verification can be more convenient for applications.
\end{remark}

\vspace*{-0.2cm}

 \section{Lagrange instability of the semilinear DAE}\label{SectThNeust}

\vspace*{-0.1cm}

The theorem on the Lagrange instability of the DAE \eqref{DAE}, which gives sufficient conditions of the existence and uniqueness of solutions with a finite escape time for the Cauchy problem \eqref{DAE}, \eqref{ini}, where the initial points $(t_0,x_0)$ satisfy the consistency condition $(t_0,x_0)\in L_0$ and the corresponding components $P_1 x_0$ belong to a certain region $\Omega$, is presented below.
\begin{theorem}\label{Th_Neust1}
Let $f\in C([0,\infty)\times \Rn, \Rn)$ have the continuous partial derivative $\frac{\partial }{\partial x} f(t,x)$ on $[0,\infty)\times \Rn$,\, $\lambda A+B$ be a regular pencil of index 1 and \eqref{soglreg2} be fulfilled. Let for any $\Hat{x}_{p_2}, \Hat{\Hat x}_{p_2}\! \hm\in\! X_2$ such that $(t_*, x_{p_1}^*\! + \Hat{x}_{p_2})$, $(t_*, x_{p_1}^*\! + \Hat{\Hat x}_{p_2})\!\hm\in\! L_0$ the operator function  \eqref{funcPhi} be basis invertible on $conv\{\Hat x_{p_2},\, \Hat{\Hat x}_{p_2}\}$. Further, let there exist a region $\Omega \subset X_1$ such that $P_1x=0 \not\in \Omega$ and the component $P_1x(t)$ of each existing solution $x(t)$ with the initial point $(t_0,x_0)\in L_0$, where $P_1 x_0 \in \Omega$, remains all the time in $\Omega$. Suppose for some self-adjoint positive operator ${H\in L(X_1)}$ there exist functions $k\in C([0,\infty),{\mathbb R})$, $U\in C((0,\infty), (0,\infty))$ such that

\vspace*{-0.3cm}

\begin{equation*}
\int\limits_c^{+\infty} \frac{dv}{U(v)} <+\infty\quad (c>0),\qquad \int\limits_{t_0}^{+\infty} k(t)\, dt =\infty,
\end{equation*}

\vspace*{-0.4cm}

\begin{equation}\label{Lagr2}
\hspace*{-0.15cm}\! (HP_1 x,G^{-1}[-BP_1 x +Q_1 f(t,x)])\! \ge k(t)\, U\!\left({\textstyle \frac{1}{2}}(HP_1x,P_1x)\right)\;\: \forall\, (t,x)\in L_0 : P_1 x \in \Omega.\!
\end{equation}

Then for each initial point $(t_0,x_0)\in L_0$, where $P_1 x_0 \in \Omega$, there exists a unique solution of the Cauchy problem \eqref{DAE}, \eqref{ini} and this solution has a finite escape time.
\end{theorem}
\begin{proof}
The beginning of the proof of Theorem~\ref{Th_Neust1} coincides with the proof of Theorem~\ref{Th_Ust1} up to the following statement. For each initial point $(t_0,z_0)$ such that $(t_0,z_0,\eta (t_0,z_0)) \in  \tilde{L}_0$ there exists a unique solution $z(t)$ of the Cauchy problem for the equation \eqref{prUstreg2} on some interval $[t_0,\varepsilon )$ with the initial condition $z(t_0)=z_0$. Hence, for each initial point $(t_0,x_0)\in L_0$, where $x_0 =P_a z_0 +P_d \eta (t_0 ,z_0)$, there exists a unique solution $x(t) =P_a z(t)\hm+P_d \eta (t,z(t))$  of the Cauchy problem \eqref{DAE}, \eqref{ini} on $[t_0,\varepsilon )$.

Further we make the following changes.

By the condition of Theorem~\ref{Th_Neust1} there exists a region $\Omega \subset X_1$ such that $P_1x=0 \not\in \Omega $ and the component $P_1x(t)$ of each solution $x(t)$ with the initial point $(t_0,x_0)\in L_0$, where $P_1 x_0 \in \Omega$, remains all the time in $\Omega$. Taking into account that $P_1 x=P_a z$, each solution $z(t)$ of the equation \eqref{prUstreg2} starting in the region $\hat{\Omega}=\{z\in \Ra \mid P_a z \in \Omega\}=P_a^{-1} \Omega$ remains all the time in it and $z=0 \not\in \hat{\Omega}$. Introduce the function $\hat{V}(z)=\frac{1}{2} (\hat{H}z,z)$, where  $\hat{H}=P_a^* H P_a$ and $H$ is an operator from \eqref{Lagr2}. Clearly, the function $\hat{V}(z)$ is positive for all $z\in \hat{\Omega}$. It follows from \eqref{Lagr2} that
\begin{equation}\label{L2}
 (\hat{H}z,P_a^{-1} G^{-1}[-BP_a z + g(t,z)])\ge k(t)\, U(\hat{V})\quad \forall\, t \ge 0,\: z \in \hat{\Omega},
\end{equation}
where $k\in C([0,\infty), {\mathbb R})$, $U\in C((0,\infty), (0,\infty))$  such that

\noindent{\centering$\int\limits_c^{+\infty} \frac{dv}{U(v)} <+\infty$, $\int\limits_{t_0}^{+\infty} k(t)\, dt \hm=\infty$.

}
Taking into account \eqref{L2}, for all $t\ge 0$ and all $z \in \hat{\Omega}$ the derivative of $\hat{V}$ along the trajectories of \eqref{prUstreg2} satisfies the following estimate:
\[
\left. \dot{\hat{V}}\right|_{\eqref{prUstreg2}} =(\hat{H}z,P_a^{-1} G^{-1}[-BP_a z + g(t,z)])\ge  k(t)\, U(\hat{V}).
\]
It follows from the properties of the functions $k$, $U$ that the inequality $\dot{v}\ge k(t)U(v)$, $t\ge 0$, has no positive solution defined in the future (see \cite[Ch.~4]{LaSal-Lef-eng}).  Then by \cite[Ch.~4, Thm.~XIV]{LaSal-Lef-eng} each solution $z(t)$ of \eqref{prUstreg2} satisfying the condition $z(t_0)=z_0$, where $z_0\in \hat{\Omega}$ and $(t_0,z_0,\eta (t_0,z_0))\hm\in \tilde{L}_0$, has a finite escape time, i.e., it exists on some finite interval $[t_0, T)$ and $\mathop{\lim }\limits_{t\to T-0}\| z(t)\| \hm= +\infty $. Then each function $x(t) =P_a z(t)+P_d \eta (t,z(t))$ with the corresponding initial values $(t_0,x_0)$, where $x_0 =P_a z_0+P_d \eta (t_0,z_0)$, is a solution of the Cauchy problem \eqref{DAE}, \eqref{ini} with finite escape time, i.e., the solution $x(t)$ is defined on the corresponding finite interval $[t_0, T)$ and $\mathop{\lim }\limits_{t\to T-0} \| x(t)\| = +\infty $.

Let us verify the \emph{uniqueness} of the solution $x(t)$, $t\in [t_0, T)$. It follows from what has been proved that the solution $x(t)$ is unique on some interval $[t_0,\varepsilon)$. Assume that the solution is not unique on $[t_0,T)$. Then there exist $t_* \in [\varepsilon ,T)$ and two different solutions $x(t)$, $\tilde{x}(t)$ with the common value $x_* =x(t_*)=\tilde{x}(t_*)$ such that $(t_*,x_*)\in L_0$ and $P_1 x_* \in \Omega$. Let us take the point $(t_* ,x_*)$ as an initial point, then there must be a unique solution of \eqref{DAE} on some interval $[t_*, \varepsilon _1)\subset [t_0,T)$ with the initial value $x(t_*)=x_*$, which contradicts the assumption.

The theorem is proven.
\end{proof}

 \section{Lagrange stability of the  mathematical model of a radio engineering filter}\label{MatModUst}

Let us consider the electrical circuit of a radio engineering filter represented in Fig.~\ref{MatMod1}. A voltage source $e$, nonlinear resistances $\varphi$, $\varphi _0$, $\psi$, a nonlinear conductance $h$, a linear resistance $r$, a linear conductance $g$, an inductance $L$ and a capacitance~$C$ are given.
\begin{figure}[H]
\centering\footnotesize
\includegraphics[width=5.6cm]{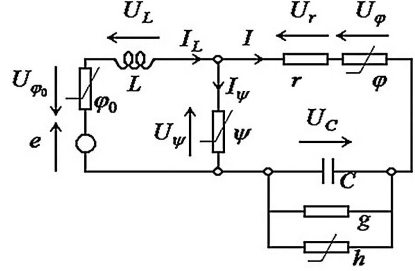}
\caption{The electric circuit diagram of the radio engineering filter}\label{MatMod1}
\end{figure}

The currents and voltages in the circuit satisfy the Kirchhoff equations, as well as the constraint equations which describe operation modes of the electric circuit elements:
\begin{center}
$I_L =I +I_\psi$, $U_\psi =U_\varphi +U_r+ U_C$, $e=U_{\varphi _0} +U_L+U_\psi$, $U_L =\frac{d(LI_L)}{dt}$, $I =\frac{d(CU_C)}{dt}+gU_C + h(U_C)$, $U_r =r I$, $U_{\varphi} =\varphi(I)$, $U_{\varphi _0} =\varphi _0(I_L)$, $U_{\psi} =\psi(I_\psi)$.
\end{center}
From these equations we obtain the system with the variables $x_1 =I_L$, $x_2 =U_C$, $x_3 =I$:
\begin{eqnarray}
 L\frac{d}{dt}x_1 + x_2 +r\, x_3 &=& e(t)-\varphi_0(x_1)-\varphi(x_3), \label{os1} \\
 C\frac{d}{dt}x_2 +g x_2 -x_3 &=& -h(x_2), \label{os2} \\
 x_2 +r x_3 &=& \psi(x_1-x_3)-\varphi(x_3). \label{os3}
\end{eqnarray}
The system describes transient process in the electrical circuit (i.e., the process of transition from one operation mode of the electric circuit to another).

It is assumed that the linear parameters $L$, $C$, $r$, $g$ are positive and real, $\varphi _0 \in C^1 ({\mathbb R})$, $\varphi \in C^1 ({\mathbb R})$, $\psi\in C^1 ({\mathbb R})$, $h \in C^1 ({\mathbb R})$ and $e \in C([0,\infty ),{\mathbb R})$.

The vector form of the system \eqref{os1}--\eqref{os3} is the semilinear DAE
\begin{equation}\label{DAE_MM}
 \frac{d}{dt}[Ax]+Bx=f(t,x),
\end{equation}
where $x=(x_1, x_2, x_3)^T=(I_L, U_C, I)^T\in {\mathbb R}^3$,
\begin{equation*}
f(t,x) \!=\!\begin{pmatrix} e(t)-\varphi_0(x_1)-\varphi(x_3) \\ -h(x_2) \\ \psi(x_1-x_3)-\varphi(x_3) \end{pmatrix}\!,\; A\!=\!\begin{pmatrix} L & 0 & 0 \\ 0 & C & 0 \\ 0 & 0 & 0 \end{pmatrix}\!,\;
B\!=\!\begin{pmatrix} 0 & 1 & r \\ 0 & g & -1 \\ 0 & 1 & r \end{pmatrix}\!.
\end{equation*}
\begin{comment}
%\noindent $x\!=\!\!\begin{pmatrix} x_1 \\ x_2 \\ x_3 \end{pmatrix}\!=\!\!\begin{pmatrix} I_L \\ U_C \\ I \end{pmatrix}\!\!\in\! {\mathbb R}^3$\!,
{\centering $f(t,x) \!=\!\begin{pmatrix} e(t)-\varphi_0(x_1)-\varphi(x_3) \\ -h(x_2) \\ \psi(x_1-x_3)-\varphi(x_3) \end{pmatrix}$\!, $A\!=\!\begin{pmatrix} L & 0 & 0 \\ 0 & C & 0 \\ 0 & 0 & 0 \end{pmatrix}$\!, %\linebreak
$B\!=\!\begin{pmatrix} 0 & 1 & r \\ 0 & g & -1 \\ 0 & 1 & r \end{pmatrix}\!\smallskip$.
}
\end{comment}

It is easy to verify that $\lambda A+B$ is a regular pencil of index 1.

The projection matrices $P_i$, $Q_i$ and the matrix $G^{-1}$ have the form
\[
P_1 =\left(\begin{array}{ccc} 1 & 0 & 0 \\ 0 & 1 & 0 \\ 0 & -r^{-1} & 0 \end{array}\right),\;
P_2 =\left(\begin{array}{ccc} 0 & 0 & 0 \\ 0 & 0 & 0 \\ 0 & r^{-1} & 1 \end{array}\right),\;
Q_1 =\left(\begin{array}{ccc} 1 & 0 & -1 \\ 0 & 1 & r^{-1} \\ 0 & 0 & 0 \end{array}\right),
\]
\[
Q_2 =\left(\begin{array}{ccc} 0 & 0 & 1 \\ 0 & 0 & -r^{-1} \\ 0 & 0 & 1 \end{array}\right),\;
G^{-1} = \left(\begin{array}{ccc} L^{-1} & 0 & -L^{-1} \\ 0 & C^{-1} & (Cr)^{-1} \\ 0 & -(Cr)^{-1} & (Cr-1)C^{-1}r^{-2} \end{array}\right).
\]

The projections of the vector $x$ have the form
\[
x_{p_1}\!=\!P_1 x\!=\!(x_1, x_2, -r^{-1}x_2)^T\!=\!(a, -rb, b)^T\!,\: x_{p_2}\!=\!P_2 x \!=\!(0, 0, r^{-1}x_2+x_3)^T \!=\!(0, 0, u)^T\!,
\]
where $a=x_1$, $b=-r^{-1}x_2$, $u=r^{-1}x_2+x_3 \in {\mathbb R}$.

The equation $Q_2[Bx-f(t,x)]=0$, determining the manifold $L_0$ from \eqref{soglreg2}, is equivalent to the equation \eqref{os3}.
Taking into account the new notation, the condition \eqref{soglreg2} holds if for any $a,\, b \in {\mathbb R}$ there exists $u \in {\mathbb R}$ such that
\begin{equation}\label{os4}
 ru =\psi(a-b-u)-\varphi(b+u).
\end{equation}

Consider the operator function $\tilde{\Phi }\colon X_2 \to L({\mathbb R}^3,Y_2)$,
\begin{center}
$\tilde{\Phi }(x_{p_2})\! =\! \left[\!\frac{\partial}{\partial x}\big(Q_2 f(t_*,x_{p_1}^*\!+x_{p_2})\big)\! -B\right]\! P_2  = \!
\big(\psi'(a_*\!-b_*\!-u)+\varphi'(b_*\!+u)+r \big)\!\! \!
\begin{pmatrix} 0 & -r^{-1} & -1 \\ 0 & r^{-2} & r^{-1} \\ 0 & -r^{-1} & -1 \end{pmatrix}$\!,
\end{center}
\begin{comment}
\begin{equation*}
\begin{split}
&\tilde{\Phi }(x_{p_2}) = \left[\frac{\partial}{\partial x}\big(Q_2 f(t_*,x_{p_1}^*+x_{p_2})\big) -B\right] P_2 = \\
&=\left( \psi'(a_*-b_*-u)+\varphi'(b_*+u)+r \right) \left(\begin{array}{ccc} 0 & -r^{-1} & -1 \\ 0 & r^{-2} & r^{-1} \\ 0 & -r^{-1} & -1 \end{array}\right),
\end{split}
\end{equation*}
\end{comment}
where $\psi'(a-b-u)= \left. \frac{d\psi(y)}{dy}\right|_{y=a-b-u}$, $\varphi'(b+u) = \left.\frac{d\varphi(y)}{dy}\right|_{y=b+u}$, $t_*\in [0,\infty )$, $a_*,b_*\in {\mathbb R}$, $x_{p_1}^*\hm=(a_*, -rb_*, b_*)^T$. Since the spaces $X_2$, $Y_2$ are one-dimensional, the invertibility of the operator function $\Phi = \left.\tilde{\Phi }\right|_{X_2}\colon X_2\hm\to L(X_2,Y_2)$ (i.e., the operator $\Phi(x_{p_2})\hm\in L(X_2,Y_2)$ is the restriction of the operator $\tilde{\Phi }(x_{p_2})\in L({\mathbb R}^3,Y_2)$ to $X_2$) is equivalent to the basis invertibility of $\Phi $. Let for any (fixed) $\Hat u,\, \Hat{\Hat u},\, a_*,\, b_*\in {\mathbb R}$ satisfying \eqref{os4} the condition $\psi'(a_*-b_*-u_*) \hm+ \varphi'(b_*+u_*)\not =-r$  be fulfilled for any $u_*\in conv \{\Hat u, \Hat{\Hat u}\}$.  Then the operator $\Lambda= \left.\tilde{\Lambda}\right|_{X_2}\hm\in L(X_2,Y_2)$, where $\tilde{\Lambda} =\tilde{\Phi }(x_{p_2}^*)$, $x_{p_2}^*=(0, 0, u_*)^T$, is invertible, since from ${\tilde{\Lambda}\, x_{p_2}=0}$, $x_{p_2}\hm\in X_2$, it follows that $x_{p_2}=0$.
Hence, for any $\Hat u,\, \Hat{\Hat u},\, a_*,\, b_*\in {\mathbb R}$ satisfying \eqref{os4} the operator function $\Phi $ \eqref{funcPhi} is basis invertible on the convex hull $conv\{\Hat x_{p_2},\Hat{\Hat x}_{p_2}\}$, where $\Hat x_{p_2}=(0, 0, \Hat u)^T$, $\Hat{\Hat x}_{p_2}=(0, 0,\Hat{\Hat u})^T$.

Choose $H =\begin{pmatrix} 2L & 0 & 0 \\ 0 & Cr & 0 \\ 0 & 0 & Cr^3 \end{pmatrix}$. Then
\begin{multline*}
\big(HP_1 x,G^{-1}[-BP_1 x +Q_1 f(t,x)]\big)= \\
=2 \big[-(g r\hm+ 1)x_2^2- x_1 \varphi_0(x_1) \hm+ (x_2\hm-x_1)\psi(x_1-x_3)-rx_2 h(x_2)- x_2 \varphi (x_3) + x_1e(t)\big].
\end{multline*}

Since $\varphi,\, \psi \!\in\!  C^1 ({\mathbb R})$, there exists a constant $C$ such that for any fixed $\Tilde x_{p_2}=(0,0,\Tilde u)^T$, where $\Tilde u\! \in\! {\mathbb R}$, and for all $t\! \in\! [0,\infty)$, $\|x_{p_1}\|\! \le\! M$, where $M$ is a constant, %$\|x_{p_1}\|\! \hm=\!\sqrt{a^2\! +\!(r^2\!+\!1)b^2}$,
the estimate $\|Q_2f(t,x_{p_1}+\Tilde x_{p_2})\|\!\le\! \sqrt{2+r^{-2}}\! \mathop{\max }\limits_{\|x_{p_1}\| \le M}\! |\psi(a-b-\Tilde u)\hm-\varphi(b+\Tilde u)|\!\le\! C$ is fulfilled. Hence, the condition \eqref{LagrA1} is satisfied for any fixed $\Tilde x_{p_2}\!=(0,0,\Tilde u)^T$ (i.e., any fixed $\Tilde u\!\in\! {\mathbb R}$).

\subsection{Conclusions}\label{ConclStab}
By Theorem~\ref{Th_Ust1} for each initial point $(t_0,x^0)\in [0,\infty )\times {\mathbb R}^3$ ($x^0 =(x_1^0,x_2^0,x_3^0)^T$) satisfying the consistency condition (the equation \eqref{os3})
\begin{equation}\label{Consist}
x_2^0 +r x_3^0 = \psi(x_1^0-x_3^0)-\varphi(x_3^0)
\end{equation}
there exists a unique solution $x(t)$ of the Cauchy problem for the DAE \eqref{DAE_MM} with the initial condition
\begin{equation}\label{ini_MM}
x(t_0)=x^0
\end{equation}
on the whole interval $[t_0, \infty)$ if: %($x(t_0)=(I_L(t_0), U_C(t_0), I(t_0))^T$)
\begin{enumerate}[1)]
 \item  for any $a,\, b \in {\mathbb R}$ there exists $u \in {\mathbb R}$ such that \eqref{os4} is fulfilled;
 \item  for any $\Hat u,\, \Hat{\Hat u},\, a_*, b_*\!\in\! {\mathbb R}$ satisfying \eqref{os4} the condition $\psi'(a_*\!-b_*\!-u_*) \hm+ \varphi'(b_*\!+u_*)\not =-r$  is fulfilled for any $u_*\in conv \{\Hat u, \Hat{\Hat u}\}$;
 \item  for some number $R>0$ there exist functions $k\in C([0,\infty),{\mathbb R})$, $U\in C((0,\infty), (0,\infty))$ such that $\int\limits_c^{+\infty} \frac{dv}{U(v)} \hm=+\infty$ and\; $-(g r+1)x_2^2- x_1 \varphi_0(x_1)+(x_2-x_1) \psi(x_1-x_3)-rx_2 h(x_2)\hm-  x_2 \varphi (x_3)+ x_1e(t)\hm\le k(t)\, U\big(Lx_1^2+Crx_2^2\big)$\; for any $t\ge 0$, $x\!\in\!  {\mathbb R}^3$ such that \eqref{os3}, $\|P_1 x\|= \sqrt{x_1^2 +(1+r^{-2})x_2^2} \ge R$.
\end{enumerate}
If, additionally, $\int\limits_{t_0}^{+\infty} k(t)\, dt <+\infty$ and
\begin{enumerate}[4)]
 \item there exists $\Tilde u \in {\mathbb R}$ such that for any $\Tilde{\Tilde u},\, a_*,\, b_* \in {\mathbb R}$ satisfying \eqref{os4} the condition \linebreak
     $\psi'(a_*-b_*-u_*) \hm+ \varphi'(b_*+u_*)\not =-r$  is fulfilled for any $u_*\in conv\{\Tilde u,\, \Tilde{\Tilde u}\}\setminus \{\Tilde u\}$ \linebreak
     (i.e., $u_*=\alpha \Tilde{\Tilde u} + (1-\alpha) \Tilde u$, $\alpha\! \in\! (0,1]\,$),
\end{enumerate}
then for the initial points $(t_0,x^0)$ the equation \eqref{DAE_MM} is Lagrange stable.

In terms of physics it means that if the input voltage $e(t)\! \in\! C([0,\infty ),{\mathbb R})$, the nonlinear resistances $\varphi,\, \varphi _0,\, \psi \in C^1 ({\mathbb R})$ and the nonlinear conductance $h \in C^1 ({\mathbb R})$ satisfy the aforementioned conditions 1)--3), then for any initial time moment $t_0\ge 0$ and any initial values $I_L(t_0)$, $U_C(t_0)$, $I(t_0)$  satisfying $U_C(t_0)+r I(t_0) \hm= \psi(I_L(t_0)-I(t_0)) - \varphi(I(t_0))$ there exist the currents $I_L(t)$, $I(t)$ and voltage $U_C(t)$ in the circuit Fig.~\ref{MatMod1} for all $t\ge t_0$, which are uniquely determined by the initial values. The functions $I_L(t)$, $U_C(t)$ are continuously differentiable and the function $I(t)$ is continuous on $[t_0, \infty)$. The currents and voltage are bounded for all $t\ge t_0$ (Lagrange stability) if, additionally, $\int\limits_{t_0}^{+\infty} k(t)\, dt <+\infty$ and the aforementioned condition 4) is satisfied. The remaining currents and voltages in the circuit are uniquely expressed in terms of $I_L(t)$, $I(t)$, $U_C(t)$.

Let us consider \emph{the particular cases}:
\begin{eqnarray}
 \varphi_0(y)=\alpha _1 y^{2k-1},\: \varphi(y)=\alpha _2y^{2l-1},\:  \psi(y)=\alpha _3y^{2j-1},\: h(y)=\alpha _4y^{2s-1}, \label{os6}\\
  \varphi _0 (y)=\alpha _1y^{2k-1},\: \varphi(y)=\alpha _2 \sin y,\: \psi(y)=\alpha _3\sin y,\: h(y)=\alpha _4 \sin y, \label{os7}
\end{eqnarray}
where $k,l,j,s\in \mathbb N$, $\alpha _i >0$, $i=\overline{1,4}$, $y\in {\mathbb R}$. Note that functions of such type for nonlinear resistances and conductances are encountered in real radio engineering devices.

For the functions of the form \eqref{os6} and each initial point $(t_0,x^0)$ satisfying \eqref{Consist}, there exists a unique solution of the Cauchy problem \eqref{DAE_MM}, \eqref{ini_MM} on $[t_0, \infty )$ if $j\le k$, $j\le s$ and $\alpha _3$ is sufficiently small. For the functions of the form \eqref{os7} and each initial point $(t_0,x^0)$ satisfying \eqref{Consist}, there exists a unique solution of the Cauchy problem \eqref{DAE_MM}, \eqref{ini_MM} on $[t_0, \infty )$ if ${\alpha _2 + \alpha _3 < r}$.  If, additionally, $\mathop{\sup }\limits_{t\in [0,\infty )}\! |e(t)| <\! +\infty$ or $\int\limits_{t_0}^{+\infty}\! |e(t)|\, dt <\!+\infty$, then for the initial points $(t_0,x^0)$ the DAE \eqref{DAE_MM} is Lagrange stable (in both cases), i.e., every solution of the DAE is bounded.
In particular, these requirements are fulfilled for voltages of the form
\begin{equation}\label{os8}
e(t)=\beta(t+\alpha)^{-n},\; e(t)=\beta e^{-\alpha t},\; e(t)=\beta e^{-\frac{(t-\alpha)^2}{\sigma^2}},\; e(t)=\beta\sin(\omega t + \theta),
\end{equation}
where $\alpha> 0$, $\beta,\, \sigma,\, \omega \in {\mathbb R}$, $n \in {\mathbb N}$, $\theta \in [0, 2\pi]$. For voltage having the form
\begin{equation}\label{os9}
e(t)= \beta(t+\alpha)^n,\; \alpha,\,\beta \in {\mathbb R},\; n \in {\mathbb N},
\end{equation}
\noindent global solutions exist, but they are not bounded on the whole interval $[t_0, \infty )$.

\subsection{Numerical analysis}\label{NumStab}

We find approximate solutions of the DAE \eqref{DAE_MM} (the system \eqref{os1}--\eqref{os3}) with the initial condition \eqref{ini_MM} using the numerical method presented in \cite{Fil.num_meth}.

Choose the parameters $L = 500$, $C = 0.5$, $r =2$, $g =0.2$ and the input voltage $e(t) \hm= 100\, e^{-t} \sin (5\,t)$. For the nonlinear resistances and conductance of the form \eqref{os6} with  $k=l=j\hm=s\hm=2$, $\alpha _i=1$, $i=\overline{1,4}$, the numerical solution with the initial values $t_0 = 0$, $x^0 =(0,0,0)^T$ is obtained. The components of the obtained solution are shown in Fig.~\ref{osf7_9}.
\begin{figure}[H]%
\centering\footnotesize
\subfloat[(a)][The current $I_L(t)$]{\includegraphics[width=6.7cm]{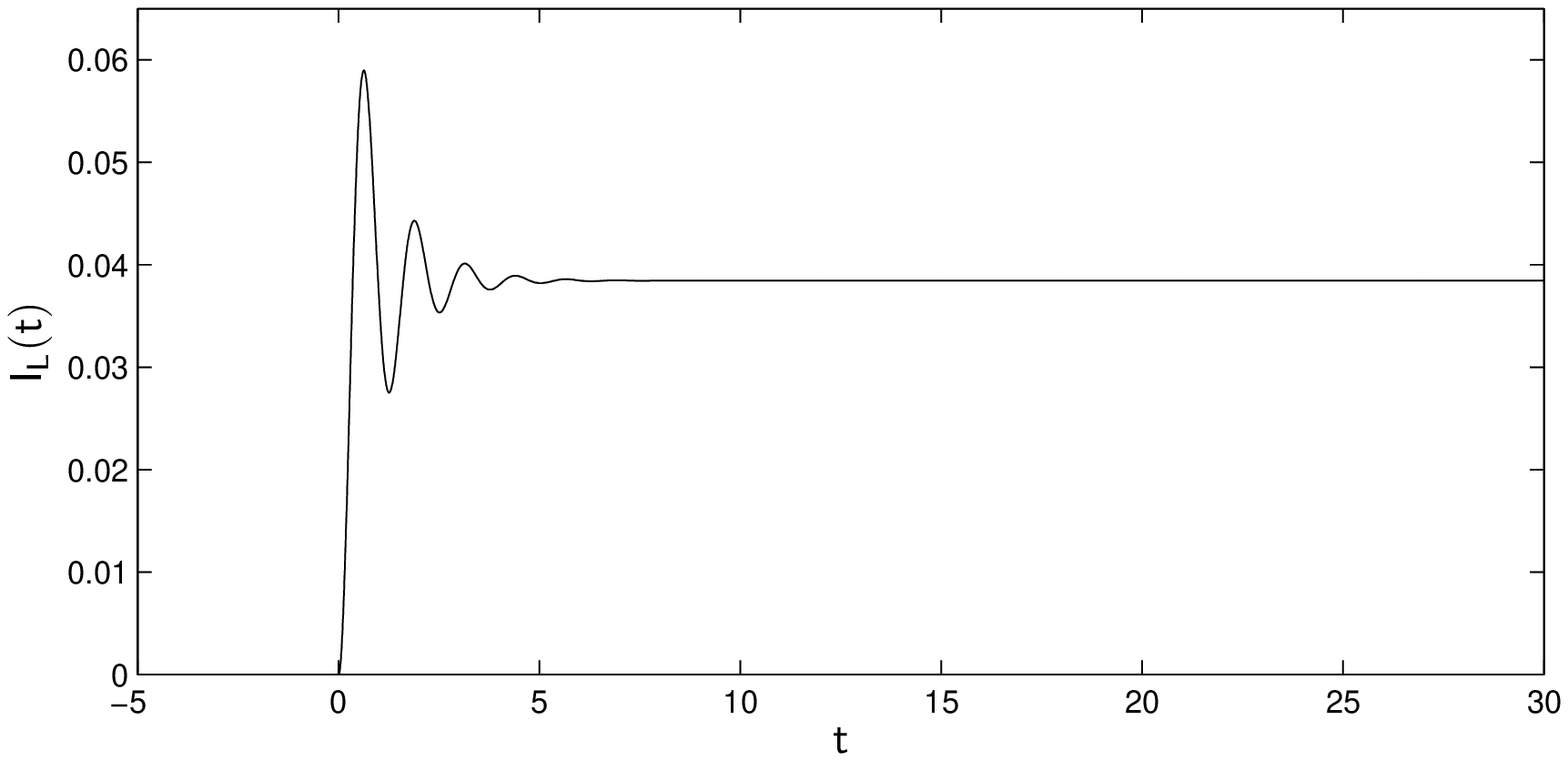}}
\subfloat[(b)][The voltage $U_C (t)$]{\includegraphics[width=6.7cm]{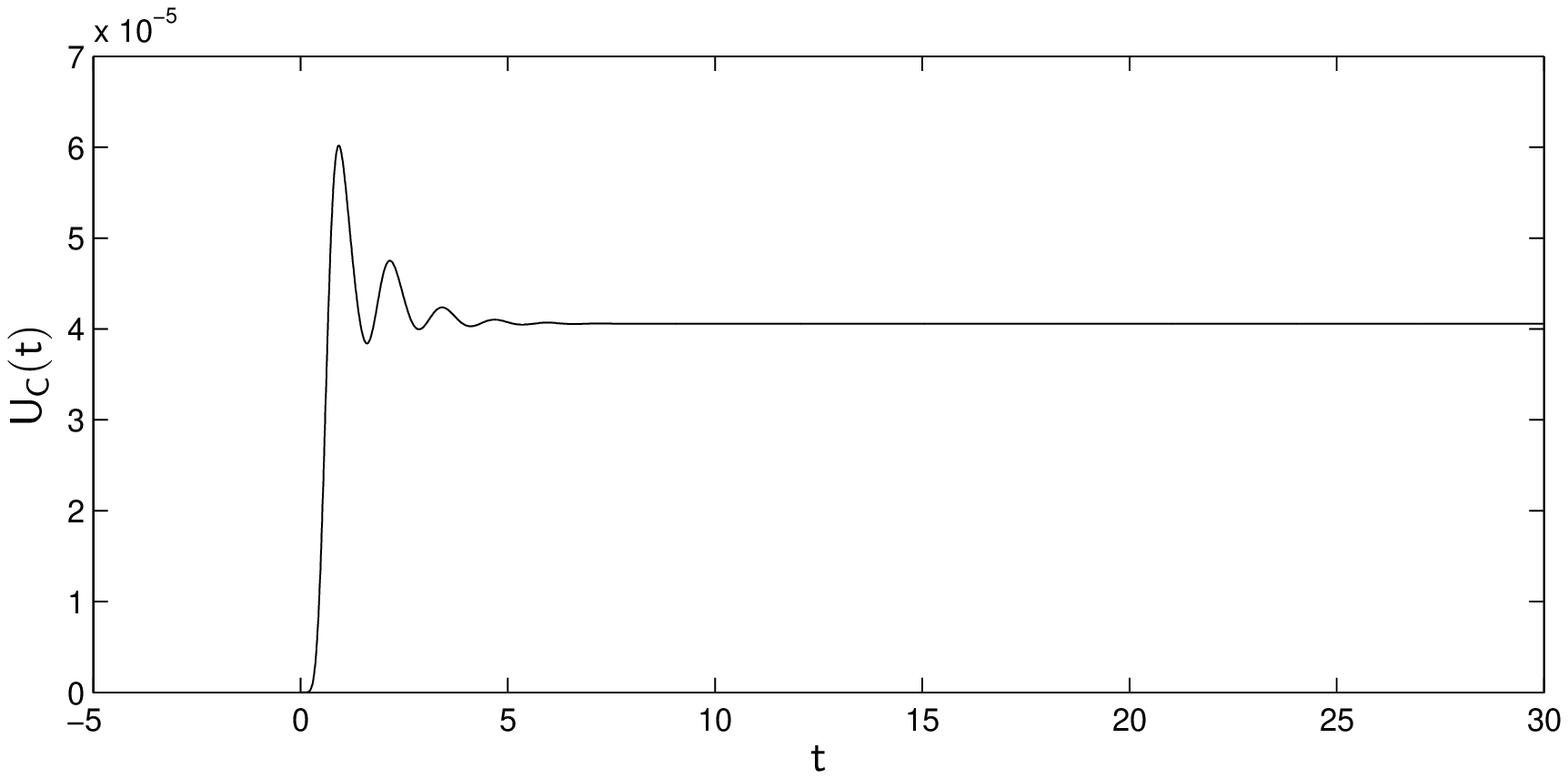}}
\end{figure}

\begin{figure}[H]%
\centering\footnotesize
\subfloat[(c)][The current $I(t)$]{\includegraphics[width=6.7cm]{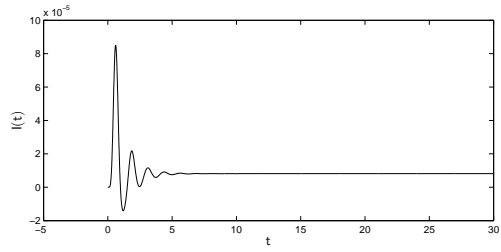}}
\caption{(a)--(c)\: The components of the numerical solution.}\label{osf7_9}
\end{figure}

The components of the solution for the electrical circuit with the linear parameters $L \hm= 50$, $C = 1$, $r = 0.001$, $g=1$, the nonlinear parameters \eqref{os6}, where  $k=l=j=s=2$, $\alpha _i=1$, $i=\overline{1,3}$, $\alpha _4=0.01$, and the input voltage $e(t) = 2 \sin t$, and for the initial values $t_0 = 0$, $x^0 =(0,0,0)^T$ are shown in Fig.~\ref{5_43__5_45}.
\begin{figure}[H]%
\centering\footnotesize
\subfloat[(a)][The current $I_L(t)$]{\includegraphics[width=6.7cm]{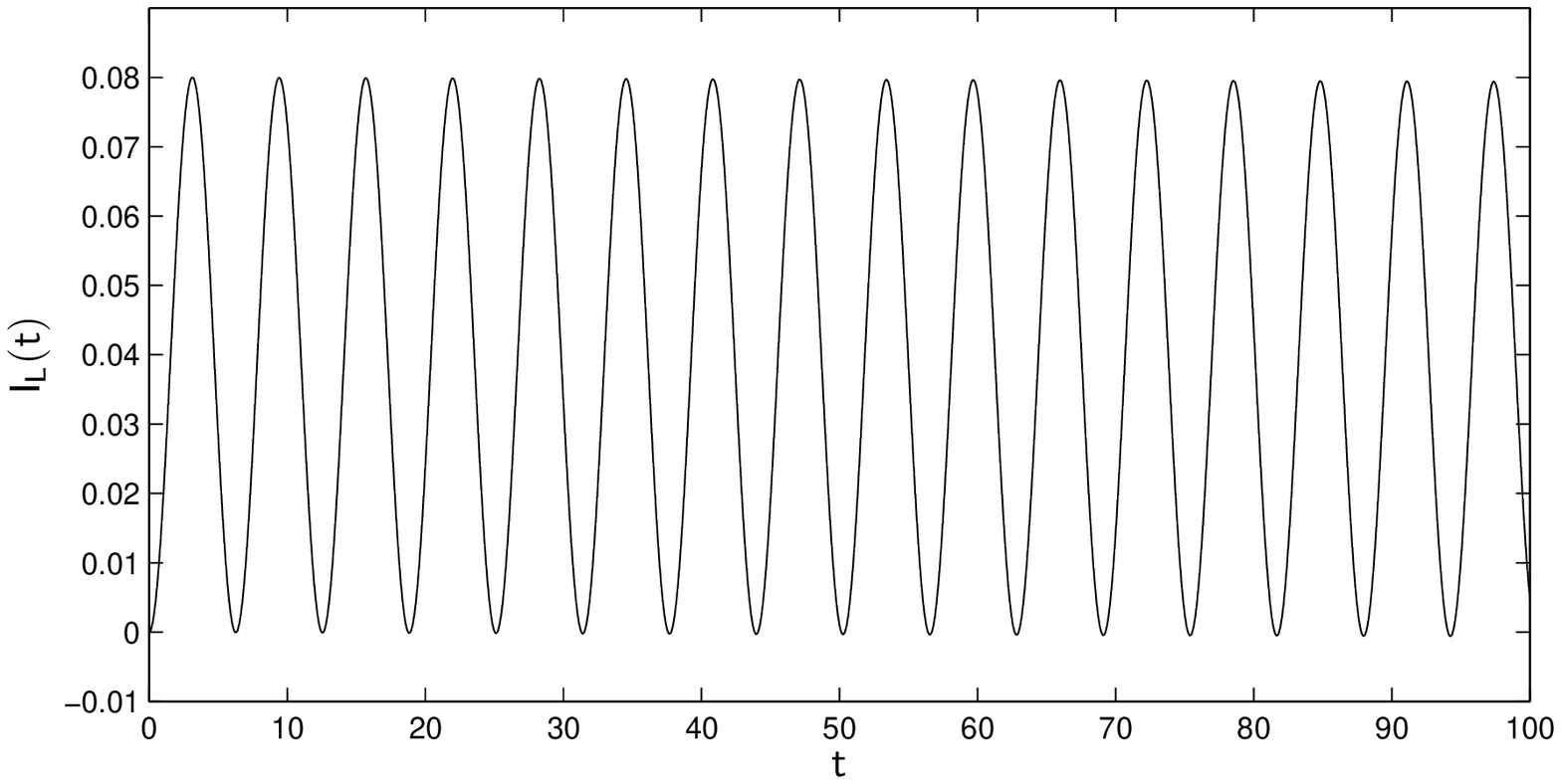}}
\subfloat[(b)][The voltage $U_C (t)$]{\includegraphics[width=6.7cm]{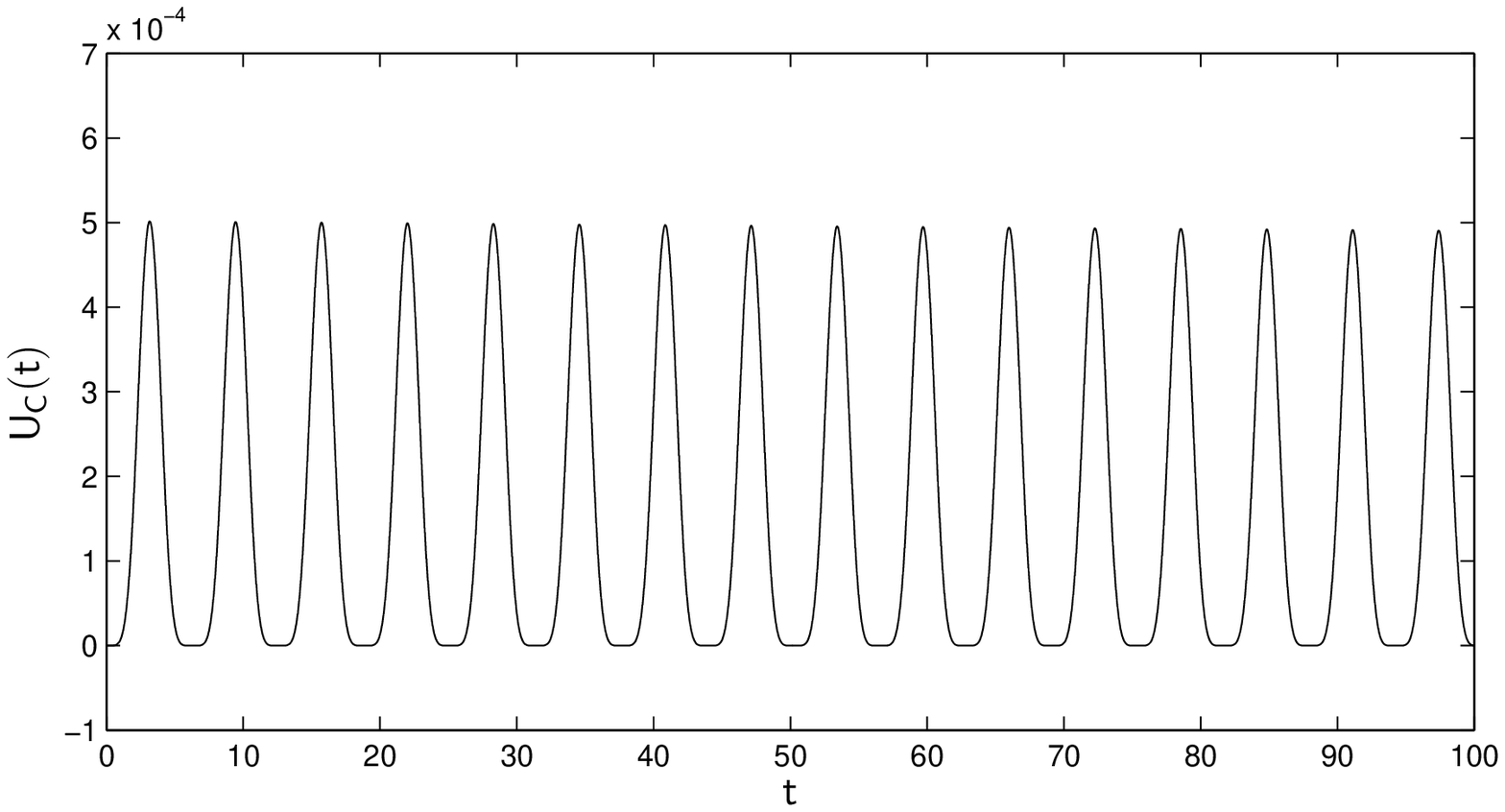}}
\end{figure}

\begin{figure}[H]%
\centering\footnotesize
\subfloat[(c)][The current $I(t)$]{\includegraphics[width=6.7cm]{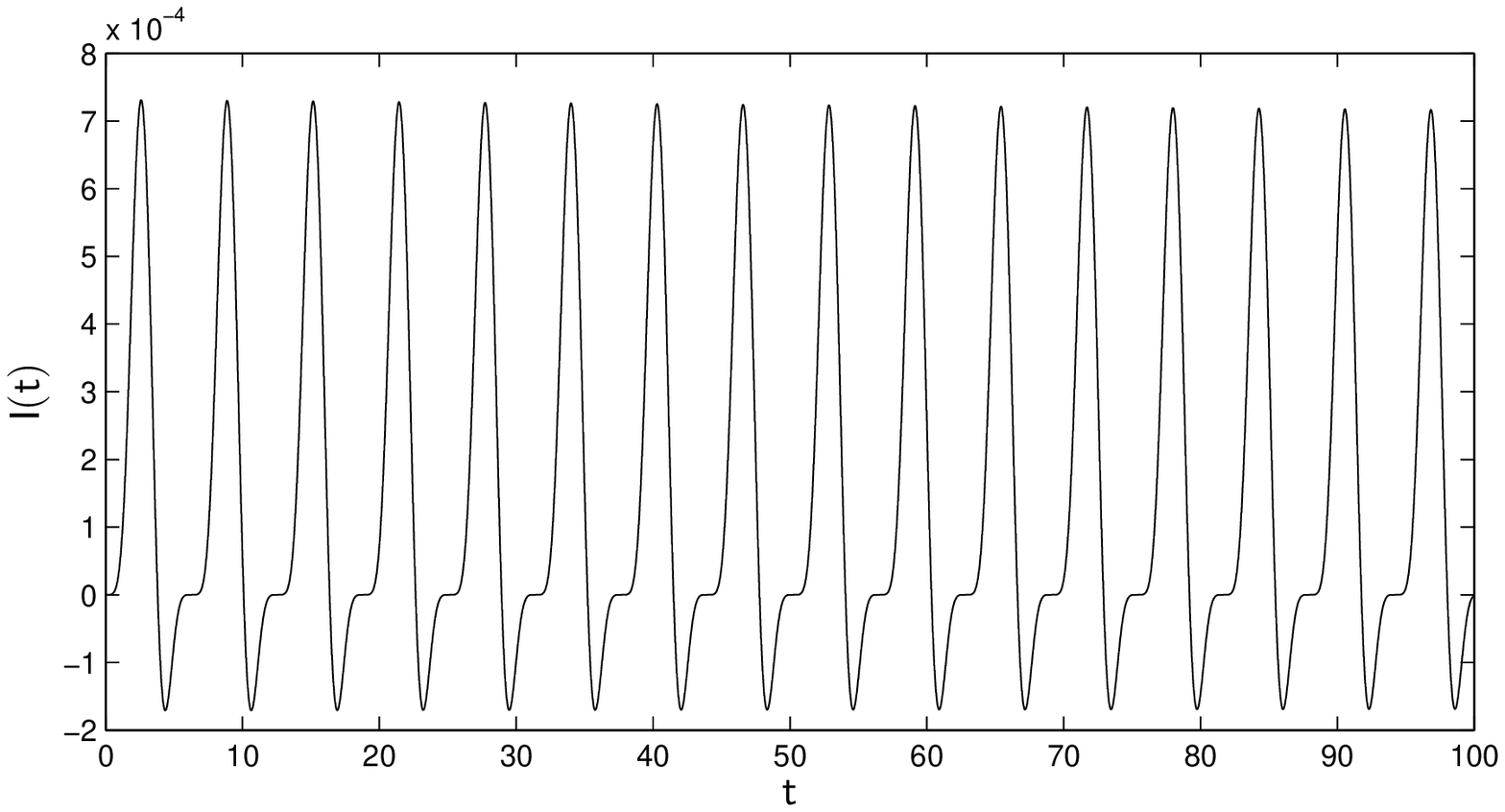}}
\caption{(a)--(c)\: The components of the numerical solution.}\label{5_43__5_45}
\end{figure}

For the linear parameters $L = 300$, $C = 0.5$, $r = 2.6$, $g =0.2$, the nonlinear resistances and conductance \eqref{os7}, where $k=2$, $\alpha _1=0.5$, $\alpha _2= 1.5$, $\alpha _3= 1$, $\alpha _4= 3$, and the voltage $e(t)= 200\, \sin(0.5\, t)-0.2$, the solution components with the initial values $t_0=0$,  $x^0 =(\pi /6, 0.5, 0)^T$ are shown in Fig.~\ref{5_73__5_75}.
\begin{figure}[H]%
\centering\footnotesize
\subfloat[(a)][The current $I_L(t)$]{\includegraphics[width=6.75cm]{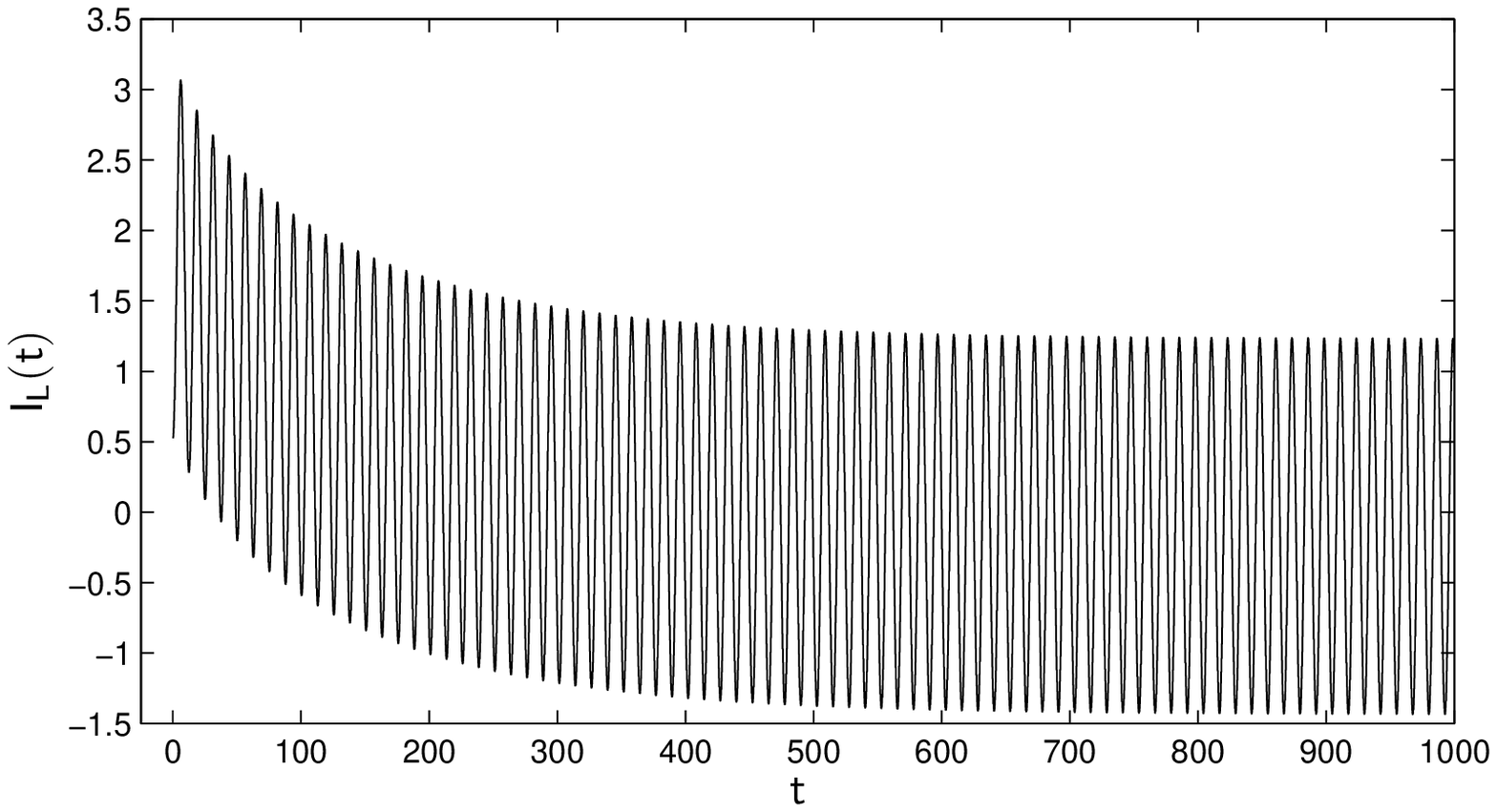}}
\subfloat[(b)][The voltage $U_C (t)$]{\includegraphics[width=6.75cm]{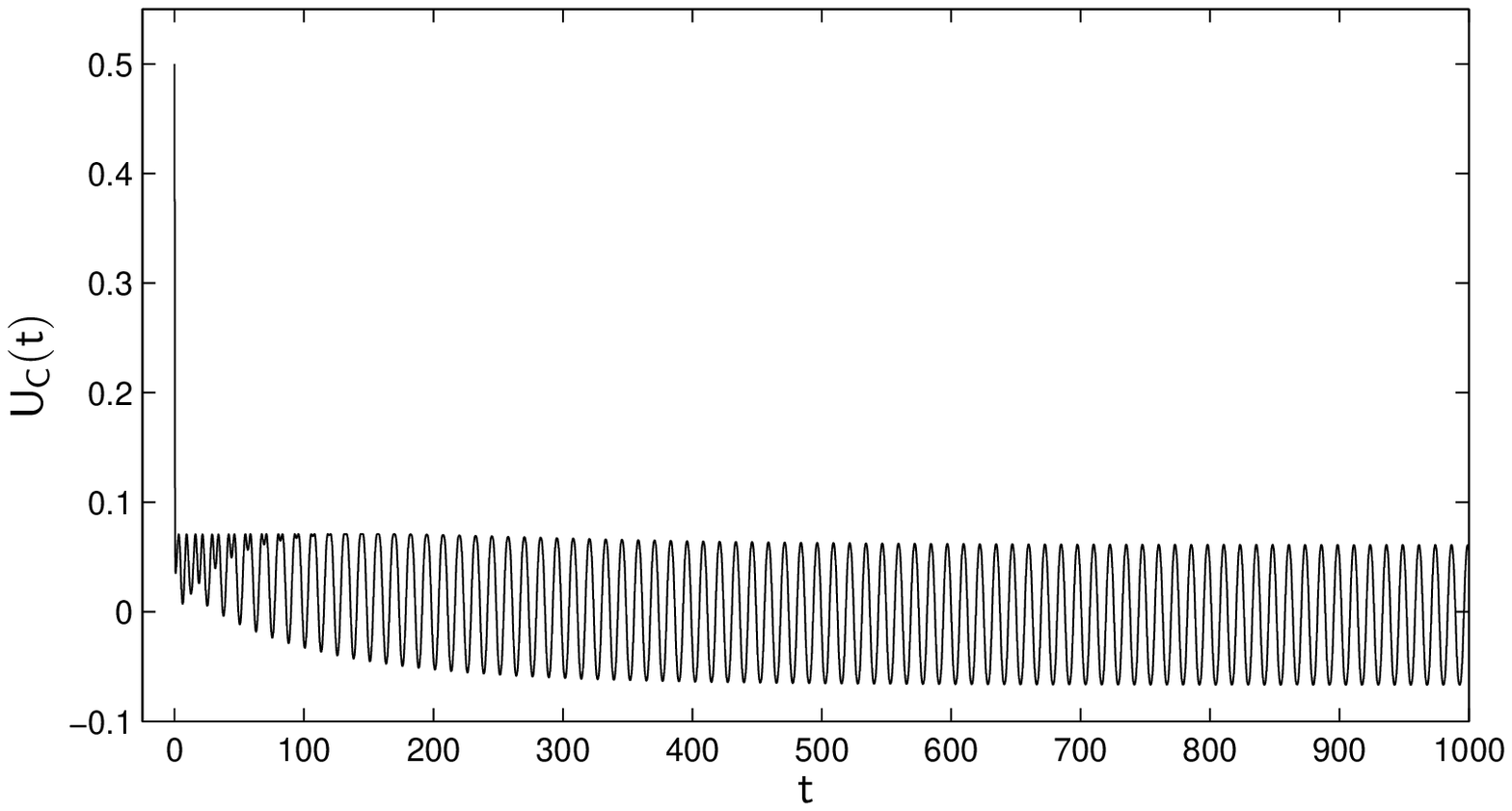}}
\end{figure}

\begin{figure}[H]%
\centering\footnotesize
\subfloat[(c)][The current $I(t)$]{\includegraphics[width=6.8cm]{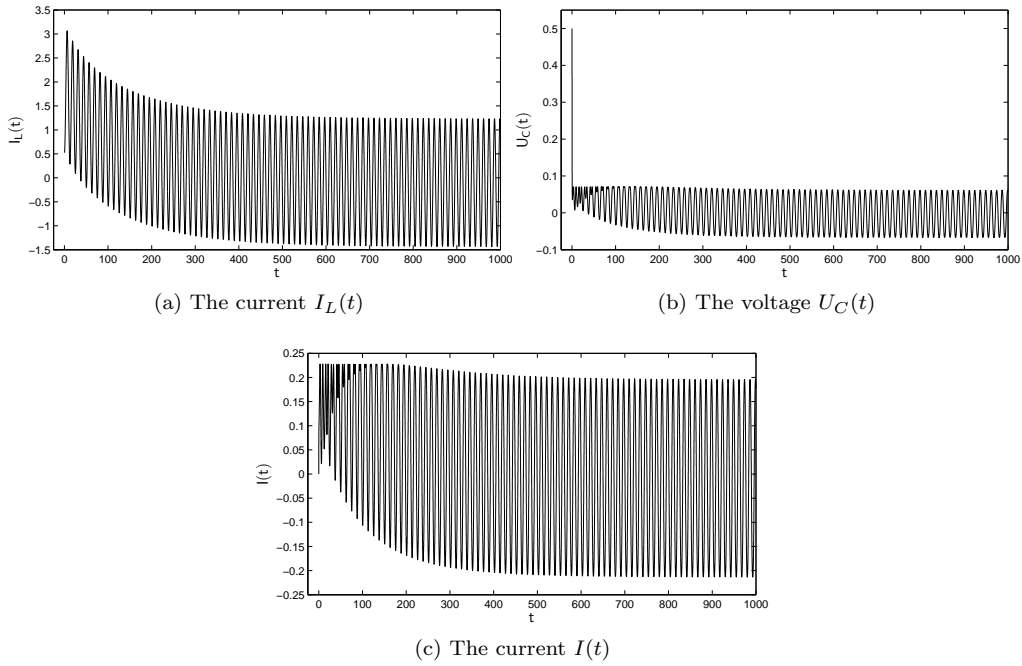}}
\caption{(a)--(c)\: The components of the numerical solution.}\label{5_73__5_75}
\end{figure}

For the linear parameters $L = 1$, $C = 5$, $r = 1.51$, $g \hm=5$, the nonlinear parameters \eqref{os7}, where $k=2$, $\alpha _i=1$, $i=1,2,4$, $\alpha _3=0.5$, the voltage $e(t) = (t+30)^{-2}$ and the initial values $t_0 = 0$, $x^0 =(0, 0, 0)^T$ the solution components are shown in Fig.~\ref{5_67__5_69}.
\begin{figure}[H]%
\centering\footnotesize
\subfloat[(a)][The current $I_L(t)$]{\includegraphics[width=6.7cm]{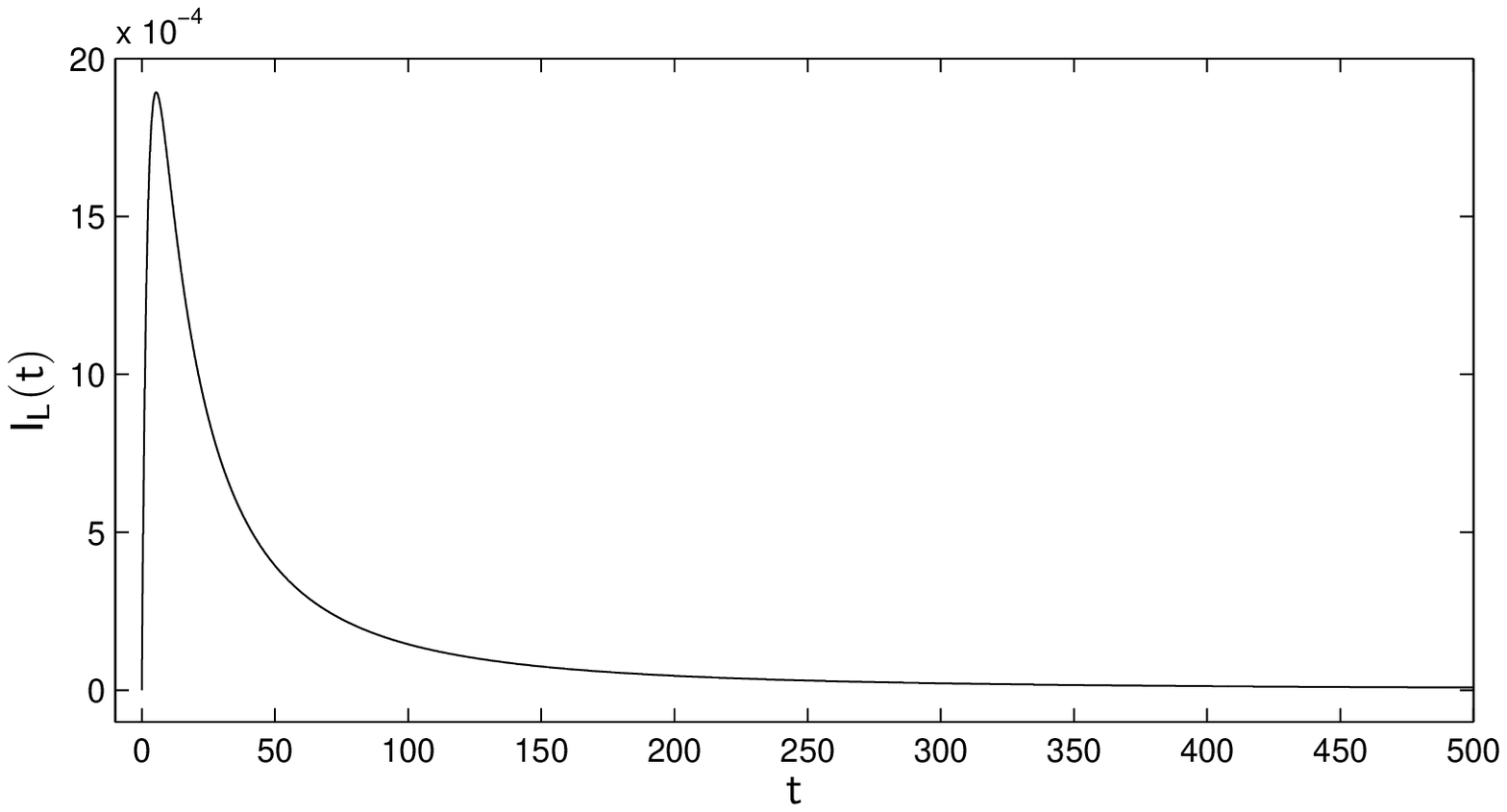}}
\subfloat[(b)][The voltage $U_C (t)$]{\includegraphics[width=6.7cm]{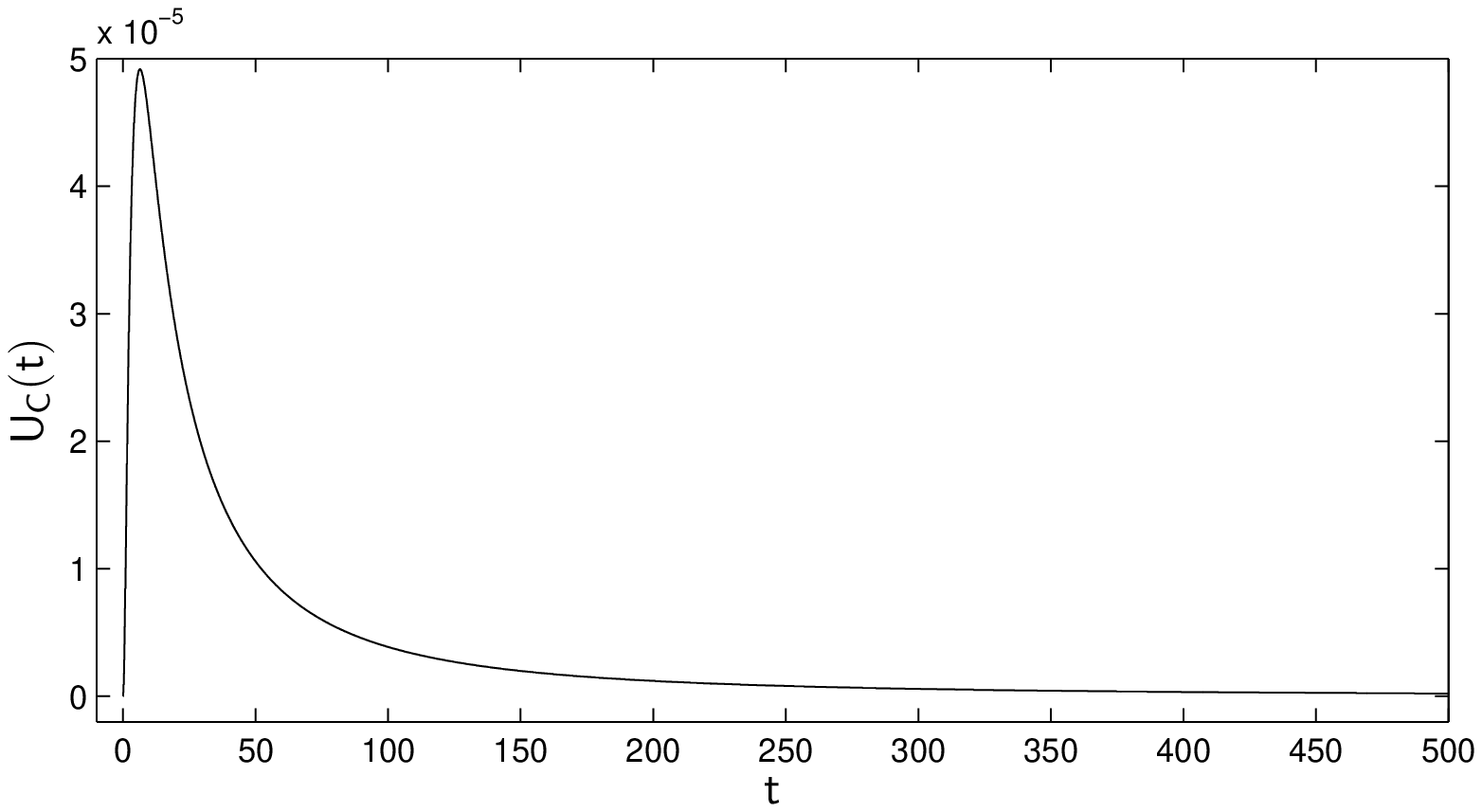}}
\end{figure}

\begin{figure}[H]%
\centering\footnotesize
\subfloat[(c)][The current $I(t)$]{\includegraphics[width=6.7cm]{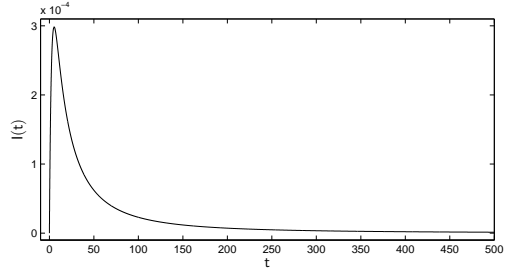}}
\caption{(a)--(c)\: The components of the numerical solution.}\label{5_67__5_69}
\end{figure}

The components of the solution for the electrical circuit with the linear parameters ${L = 1000}$, $C = 0.5$, $r =2$, $g =0.3$, the nonlinear parameters \eqref{os6} with $k=l=j=s=2$, $\alpha _i=1$, $i=\overline{1,4}$, the input voltage $e(t) =-t^2$, and for the initial values $t_0 = 0$, $x^0 =(0,0,0)^T$ are shown in Fig.~\ref{osglf1_f3}.
\begin{figure}[H]%
\centering\footnotesize
\subfloat[(a)][The current $I_L(t)$]{\includegraphics[width=6.7cm]{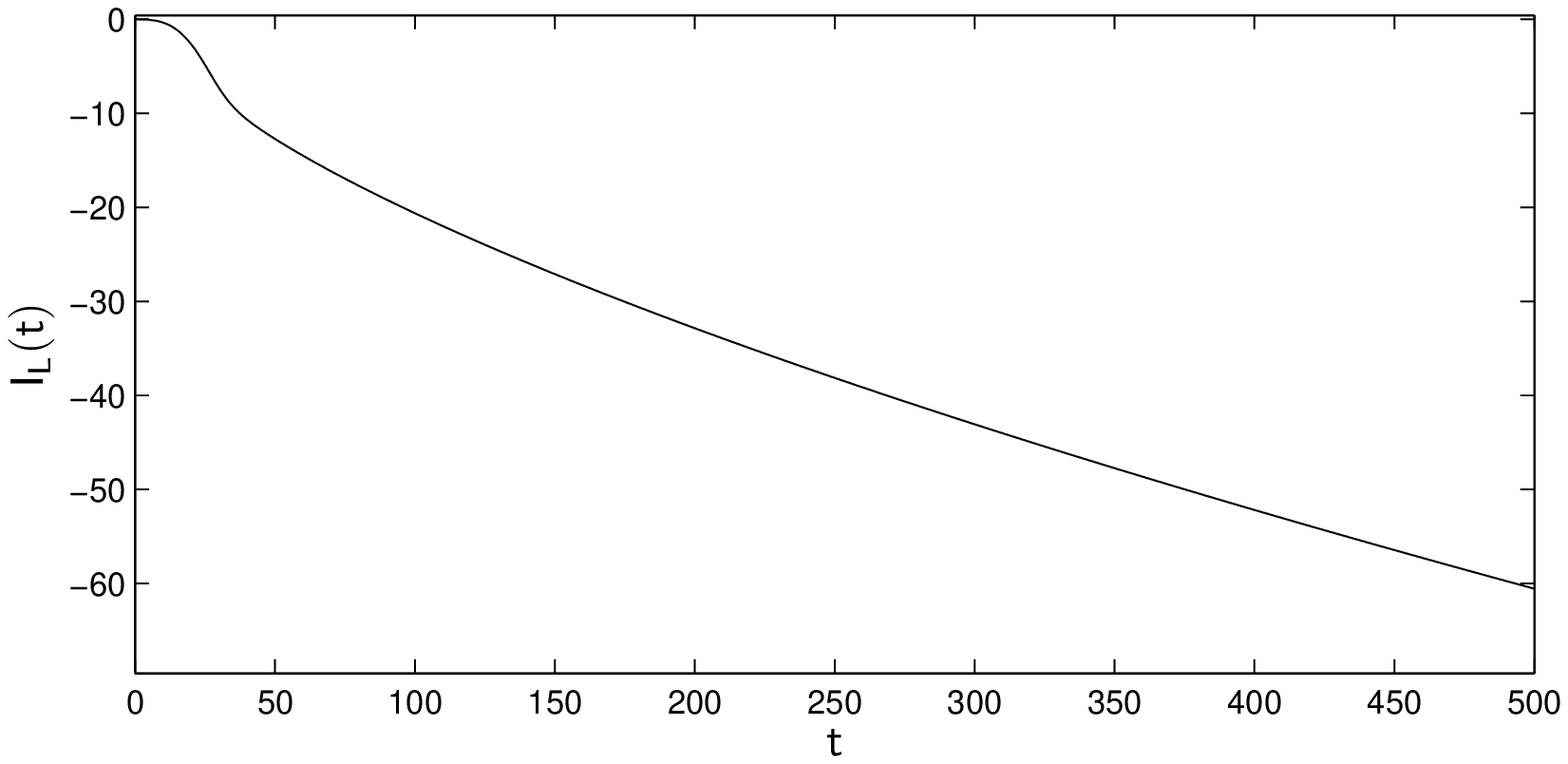}}
\subfloat[(b)][The voltage $U_C (t)$]{\includegraphics[width=6.7cm]{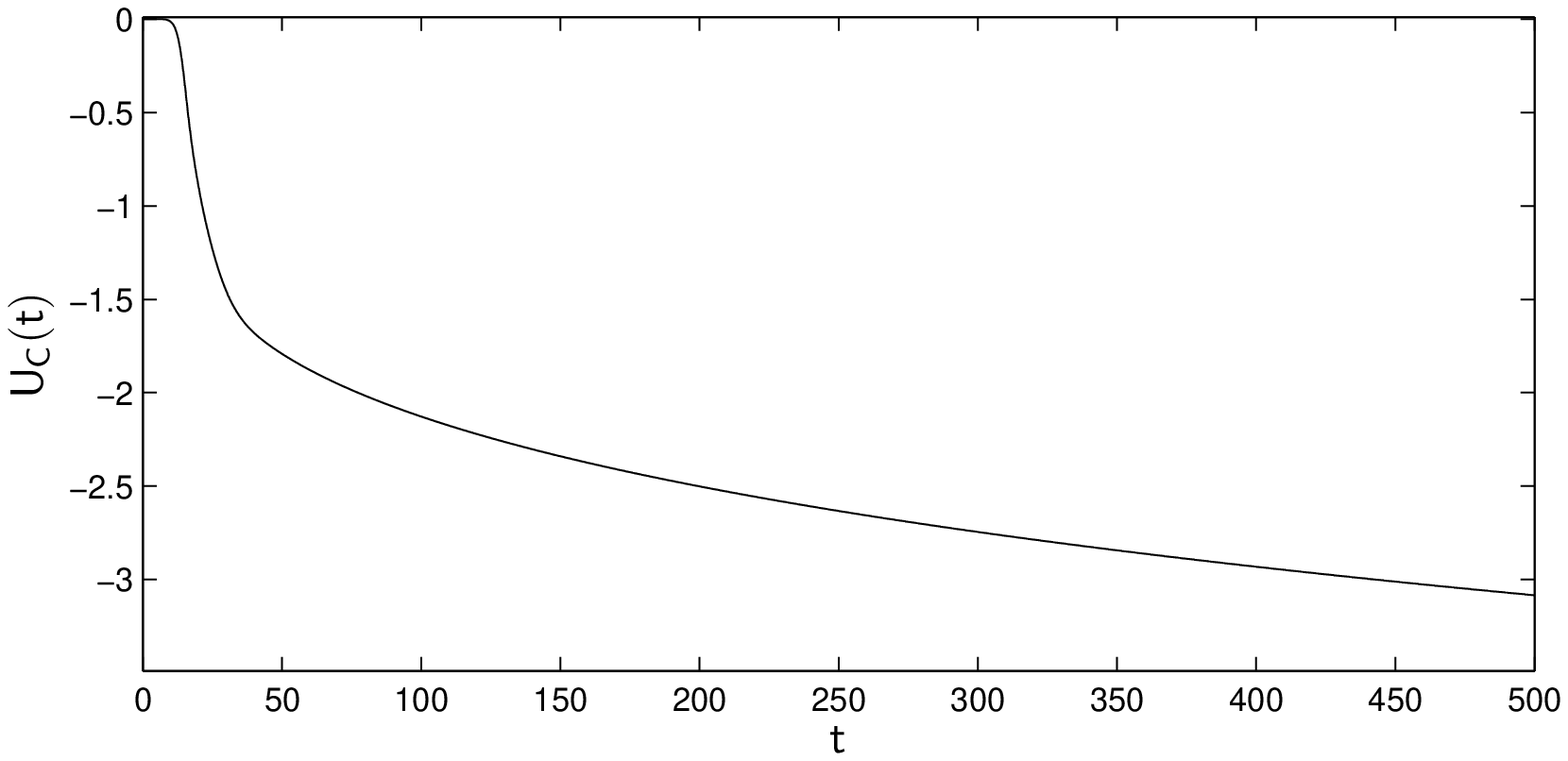}}
\end{figure}

\begin{figure}[H]%
\centering\footnotesize
\subfloat[(c)][The current $I(t)$]{\includegraphics[width=6.7cm]{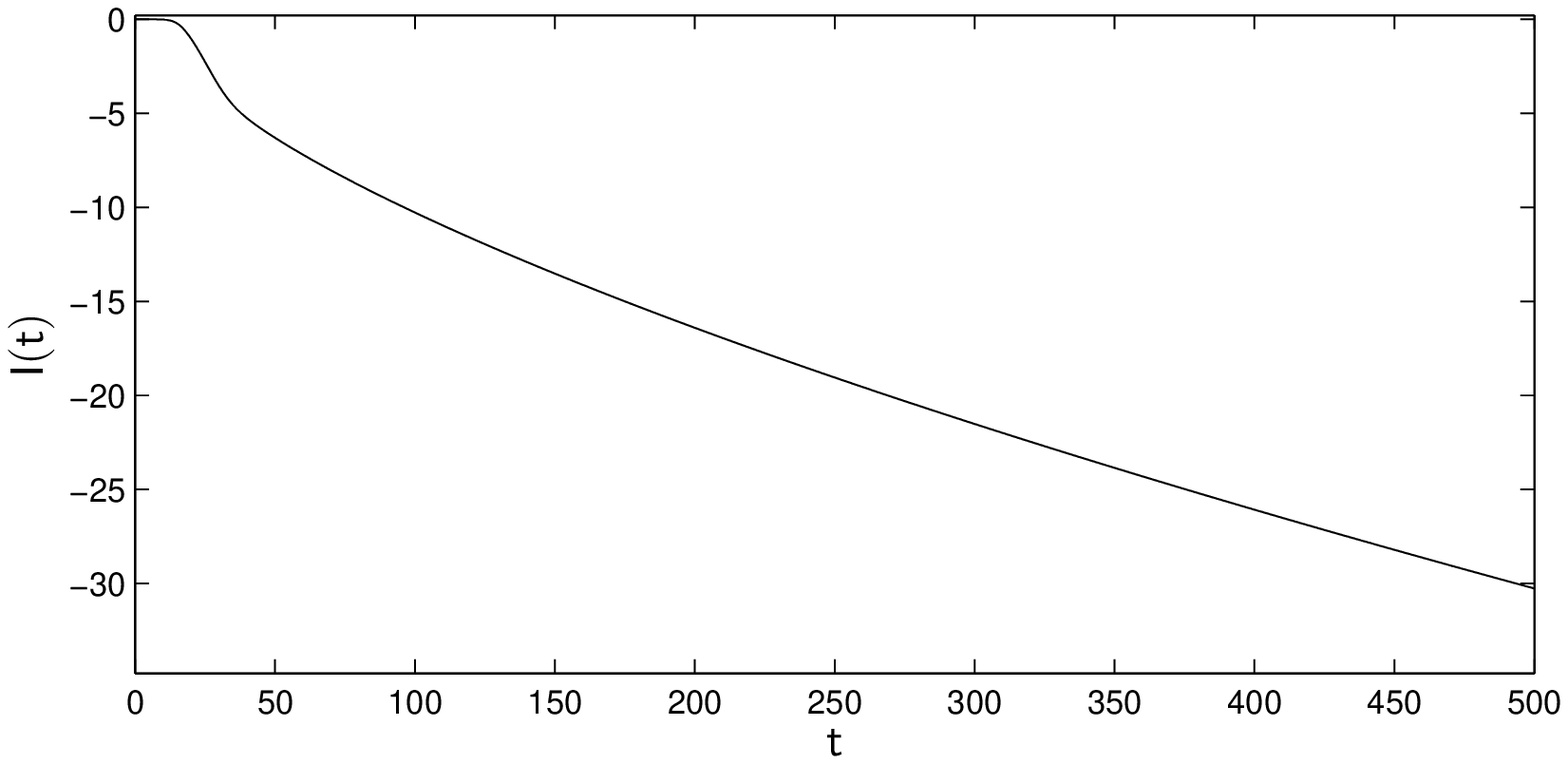}}
\caption{(a)--(c)\: The components of the numerical solution.}\label{osglf1_f3}
\end{figure}

For the linear parameters $L = 100$, $C = 5$, $r = 3$, $g =4$, the nonlinear parameters \eqref{os7}, where $k=2$, $\alpha _1=1$, $\alpha _2= 0.9$,  $\alpha _3=2$,  $\alpha _4=5$, the voltage $e(t) = (t-50)^3$ and the initial values $t_0 = 0$, $x^0 =(0, 0, 0)^T$ the solution components are shown in Fig.~\ref{5_76__5_78}.
\begin{figure}[H]%
\centering\footnotesize
\subfloat[(a)][The current $I_L(t)$]{\includegraphics[width=6.7cm]{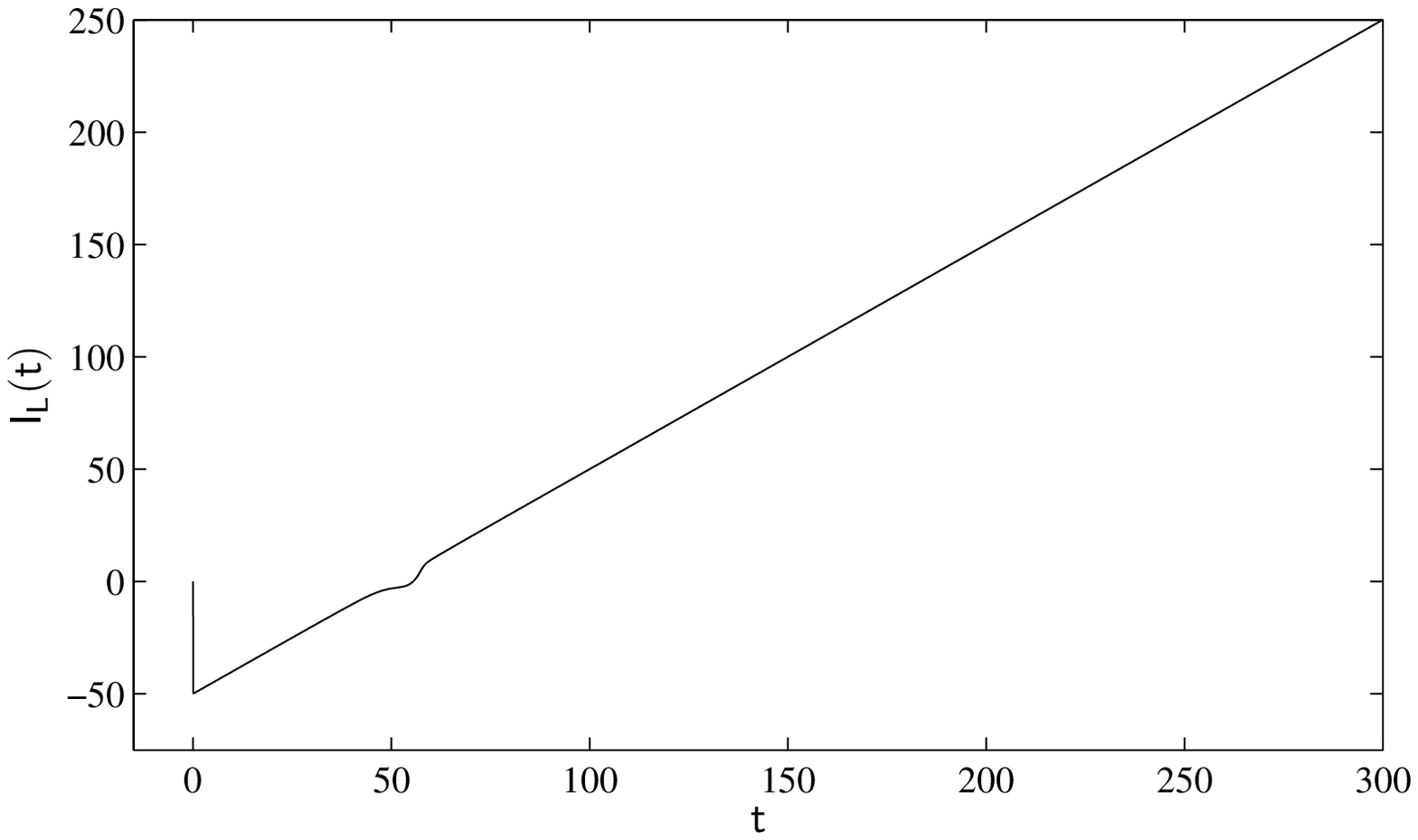}}
\subfloat[(b)][The voltage $U_C (t)$]{\includegraphics[width=6.8cm]{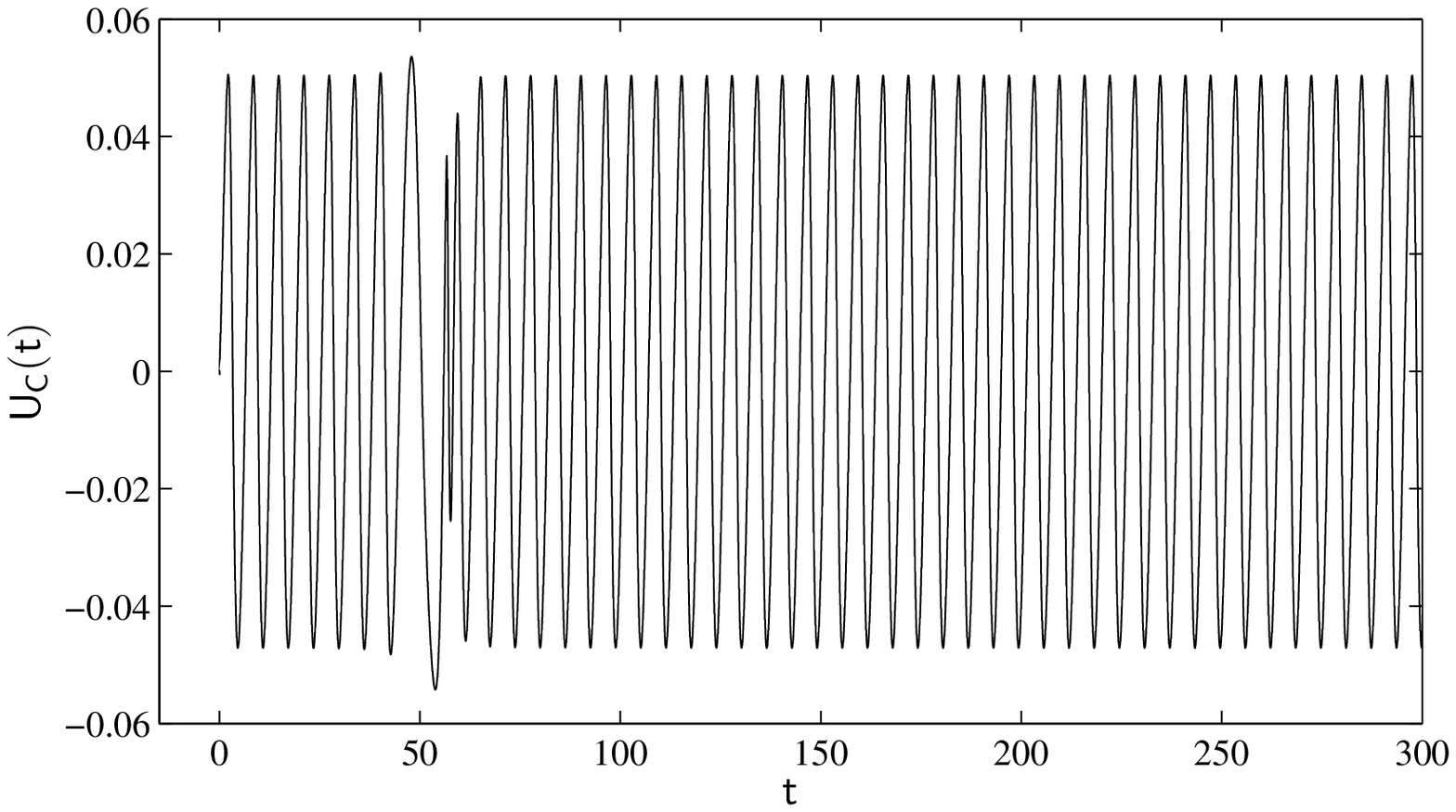}}
\end{figure}

\begin{figure}[H]%
\centering\footnotesize
\subfloat[(c)][The current $I(t)$]{\includegraphics[width=6.8cm]{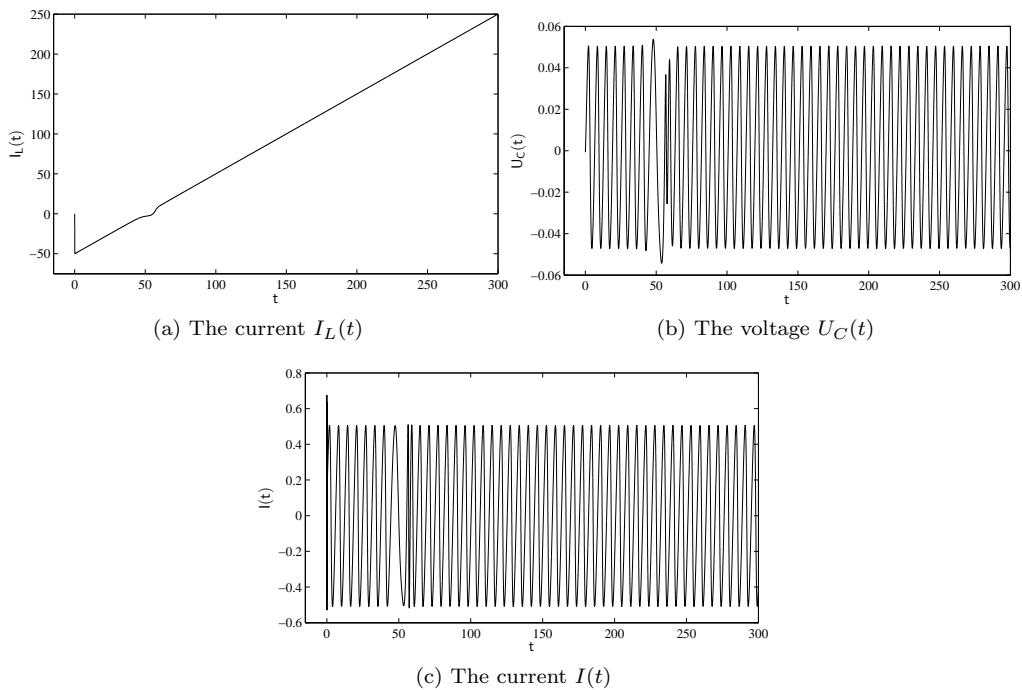}}
\caption{(a)--(c)\: The components of the numerical solution.}\label{5_76__5_78}
\end{figure}

The numerical solutions shown in Fig.~\ref{osf7_9}--\ref{5_67__5_69} are bounded on the corresponding time intervals. When we increase the time intervals by a factor of 5--10, the solutions are bounded similarly.
The analysis of these numerical solutions indicates that there exist bounded global solutions of the equation \eqref{DAE_MM} (the system \eqref{os1}--\eqref{os3}) with the input voltage of the form \eqref{os8} and the nonlinear resistances and conductance of the form \eqref{os6},~\eqref{os7}. The analysis of the numerical solutions shown in Fig.~\ref{osglf1_f3},~\ref{5_76__5_78} indicates that there exist global solutions, increasing without bound with an increase in time (as $t \to \infty$), for the equation \eqref{DAE_MM} (the system \eqref{os1}--\eqref{os3}) with the input voltage of the form \eqref{os9} and the nonlinear parameters of the form \eqref{os6},~\eqref{os7}. Similar results follow from the application of Theorem~\ref{Th_Ust1}. Therefore, the conclusions obtained with the help of this theorem are verified by a numerical experiment.

%\vspace*{-0.2cm}

 \section{Lagrange instability of the  mathematical model of a radio engineering filter}\label{MatModNeust}

Consider the system \eqref{os1}--\eqref{os3} (the DAE \eqref{DAE_MM}) with the nonlinear resistances and conductance
\begin{equation}\label{osNeust1}
\varphi_0(x_1)=-x_1^2,\; \varphi(x_3)=x_3^3,\; \psi(x_1-x_3)=(x_1-x_3)^3,\;  h(x_2)=x_2^2.
\end{equation}
It is assumed that there exists $M_e = \mathop{\sup }\limits_{t\in [t_0,\infty )} |e(t)| < +\infty$.

The verification of the condition \eqref{soglreg2} and the condition for the operator function~\eqref{funcPhi} is similar to the verification, which  has been carried out in Section~\ref{MatModUst}, and, it is easy to verify that these requirements are fulfilled.

Denote $z=(x_1,x_2)^T\in {\mathbb R}^2$. Choose
\begin{equation}\label{Omega}
 \begin{split}
 \Omega_{{\mathbb R}^2}\! =&\!\biggl\{z=\begin{pmatrix} x_1 \\ x_2 \end{pmatrix}\!\in {\mathbb R}^2 \,\Big{|}\, x_1 > m_1,\: m_1 = \max \Bigl\{1\!+\!\sqrt{M_e},\, \sqrt[3]{g\!+\!\frac{1}{r}},\, \frac{3C}{L},  \\
 &\sqrt{\max\bigl\{\frac{L}{3rC}\!-\!\frac{r}{3}, 0\bigr\}}\, \Bigr\},\: x_2 < -rx_1 \hm-x_1^3 - m_2,\: m_2 = \max\bigl\{g-\frac{2Cr}{L},\, 0\bigr\} \biggr\},
 \end{split}
\end{equation}
\vspace*{-0.2cm}
\[
\Omega = \{x_{p_1}=P_1x\in X_1 \mid z\in  \Omega_{{\mathbb R}^2}\}.
\]
Since $x_{p_1}\!=\!P_1 x\!=\!(x_1, x_2, -r^{-1}x_2)^T$, then $x_{p_1}\!\in\! \Omega \hm\Leftrightarrow z\!\in\! \Omega_{{\mathbb R}^2}$. Obviously, $x_{p_1}\!=\!0 \!\not\in\! \Omega$.

The boundary of the region $\Omega_{{\mathbb R}^2}$ consists of the parts $x_1=m_1$ and $x_2 +rx_1 +x_1^3 +m_2\hm=0$. Since $x_1\ge m_1$,  $\frac{d}{dt}x_1 > 0$ and $x_2 +rx_1 +x_1^3 +m_2 \hm\le 0$, $\frac{d}{dt}(x_2 +rx_1 +x_1^3 +m_2) < 0$ for all $t\ge 0$, $x=(x_1,x_2,x_3)^T\! \in\! {\mathbb R}^3$ satisfying \eqref{os3} (the condition $(t,x)\!\in\! L_0$), where $z=(x_1,x_2)^T\! \in\! \overline\Omega_{{\mathbb R}^2}$ ($\overline\Omega_{{\mathbb R}^2}$ is the closure of $\Omega_{{\mathbb R}^2}$), the component $z(t)=(x_1(t),x_2(t))^T$ of each existing solution, which starts at time $t_0\ge 0$ in the region $\Omega_{{\mathbb R}^2}$, cannot leave this region. Consequently, the component $x_{p_1}(t)=P_1x(t)$ of each existing solution $x(t)$ with the initial point $(t_0,x^0)\in [0,\infty )\times {\mathbb R}^3$ ($x^0 =(x_1^0,x_2^0,x_3^0)^T$) satisfying \eqref{Consist}, where $P_1 x^0 \in \Omega$ ($(x_1^0,x_2^0)^T\in \Omega_{{\mathbb R}^2}$), remains all the time in $\Omega$.

We choose $H =\left(\begin{array}{ccc} 2L & 0 & 0 \\ 0 & C & 0 \\ 0 & 0 & Cr^2 \end{array} \right)$. Then for any $x=(x_1,x_2,x_3)^T$ satisfying \eqref{os3} and such that $(x_1,x_2)^T\in \Omega_{{\mathbb R}^2}$, the condition\, $\big(HP_1 x,G^{-1}[-BP_1 x +Q_1 f(t,x)]\big)\hm=\! 2\bigl[e(t)x_1-(g +r^{-1})x_2^2+ x_1^3+(r^{-1}x_2 -x_1)(x_1-x_3)^3 - x_2^3-r^{-1}x_2 x_3^3\bigr]\! >\!2[-(g +r^{-1})x_2^2+x_2^3]\! \hm\ge \alpha\, v^{3/2}$,
where $v=\frac{1}{2}(HP_1x,P_1x)=Lx_1^2+Cx_2^2$ and $\alpha>0$ is a certain constant, is fulfilled. Hence,
the condition \eqref{Lagr2}, where $k(t)\equiv 1$, $U(v)\hm=\alpha\, v^{3/2}$, is fulfilled.

Thus, all the conditions of Theorem~\ref{Th_Neust1} are satisfied.

\vspace*{-0.2cm}

 \subsection{Conclusions}\label{ConclInstab}

By Theorem~\ref{Th_Neust1} for each initial point $(t_0,x^0)\in [0,\infty )\times {\mathbb R}^3$ satisfying \eqref{Consist} and such that $(x_1^0,x_2^0)^T\in \Omega_{{\mathbb R}^2}$, where $\Omega_{{\mathbb R}^2}$ is the region \eqref{Omega}, there exists a unique solution of the Cauchy problem for the DAE \eqref{DAE_MM} with the initial condition \eqref{ini_MM}, where the functions $\varphi_0$, $\varphi$, $\psi$, $h$ have the form \eqref{osNeust1} and  $\mathop{\sup }\limits_{t\in [t_0,\infty )} |e(t)| < +\infty$, and this solution has a finite escape time (the solution exists on some finite interval and is unbounded).

In terms of physics it means that if $\mathop{\sup }\limits_{t\in [t_0,\infty )} |e(t)| < +\infty$ and the nonlinear resistances and conductance have the form \eqref{osNeust1}, then for any initial time moment $t_0\ge 0$ and any initial values $I_L(t_0)$, $U_C(t_0)$, $I(t_0)$ satisfying $U_C(t_0)+r I(t_0) \hm= \psi(I_L(t_0)-I(t_0)) - \varphi(I(t_0))$ and such that $(I_L(t_0),U_C(t_0))^T\in \Omega_{{\mathbb R}^2}$, on some finite interval $t_0\le t < T$ there exist the currents $I_L(t)$, $I(t)$ and voltage $U_C(t)$ in the circuit Fig.~\ref{MatMod1}, which are uniquely determined by the initial values, and $\mathop{\lim }\limits_{t\to T-0} \left\|(I_L(t), U_C(t), I(t))^T\right\|=+\infty $.

\subsection{Numerical analysis}\label{NumInstab}

We find approximate solutions for the DAE \eqref{DAE_MM} (the system \eqref{os1}--\eqref{os3}) with the functions of nonlinear resistances and conductance \eqref{osNeust1} and the initial condition \eqref{ini_MM}. Initial values  $t_0$, $x^0\! =\!(x_1^0,x_2^0,x_3^0)^T$\! are chosen so that \eqref{Consist} is satisfied and $(x_1^0,x_2^0)^T\!\!\in\! \Omega_{{\mathbb R}^2}$, where $\Omega_{{\mathbb R}^2}$ is~\eqref{Omega}.

Choose the parameters $L = 10$, $C = 0.5$, $r = 2$, $g =0.2$, the input voltage $e(t) = 2\, \sin t$ and the initial values $t_0 = 0$, $x^0 \hm=(2.45,\,  -20.625125,\, 2.5)^T$. The components of the obtained numerical solution are shown in Fig.~\ref{osnf1_f3}.
\begin{figure}[H]%
\centering\footnotesize
\subfloat[(a)][The current $I_L(t)$]{\includegraphics[width=4.5cm]{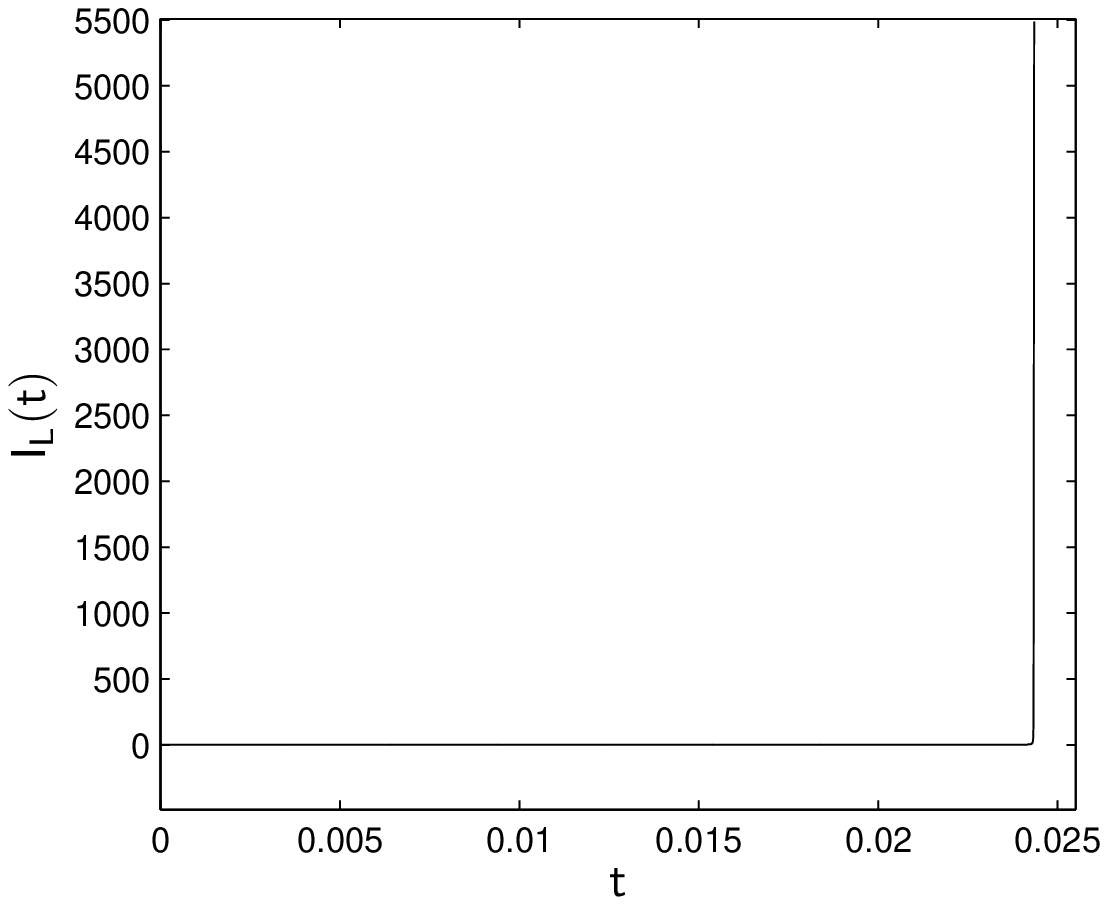}}
\subfloat[(b)][The voltage $U_C (t)$]{\includegraphics[width=4.5cm]{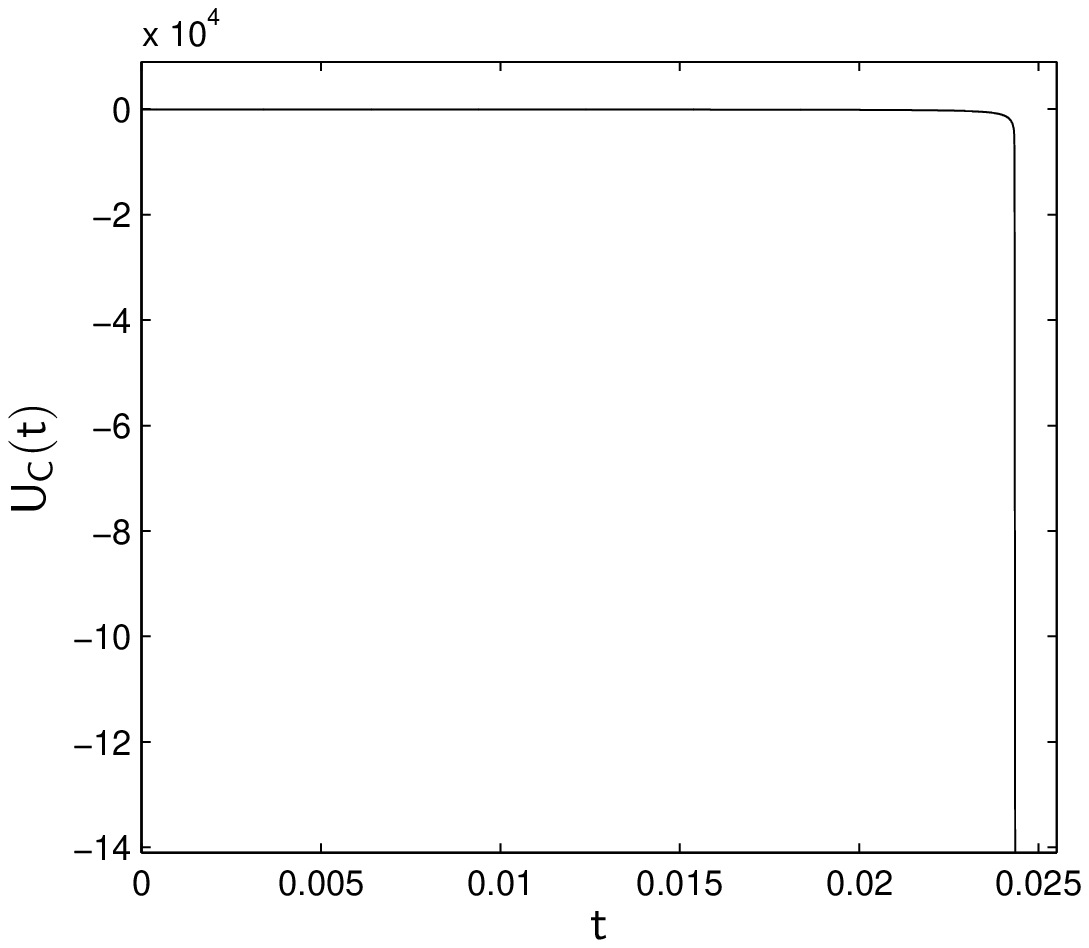}}
\subfloat[(c)][The current $I(t)$]{\includegraphics[width=4.5cm]{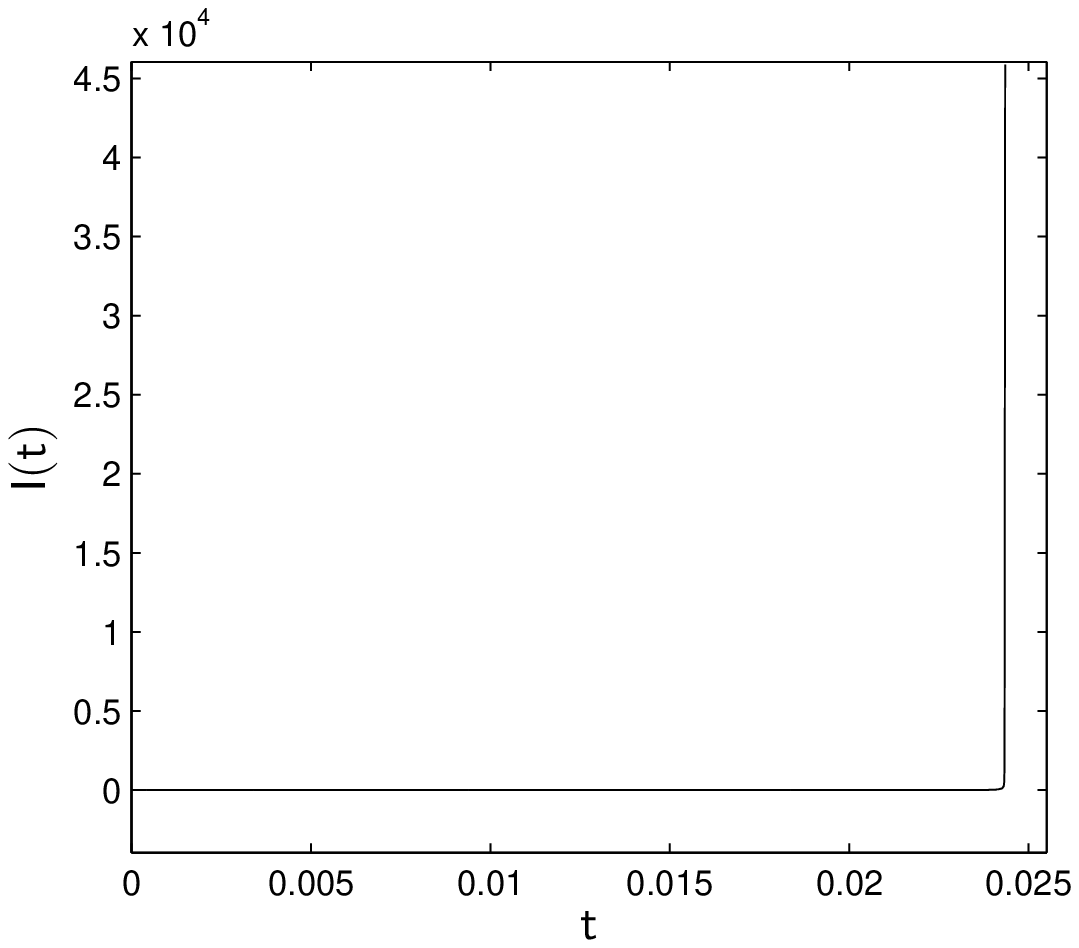}}
\caption{(a)--(c)\: The components of the numerical solution.}\label{osnf1_f3}
\end{figure}

For the electrical circuit with the linear parameters $L = 5$, $C = 0.5$, $r = 2$, $g =0.5$ and the input voltage $e(t) = 0$, the components of the numerical solution with the initial values $t_0 = 0$,  $x^0 \hm=(1.1,\,  -4.129,\, 1.2)^T$ have the form similar to that shown in Fig.~\ref{osnf1_f3}.

The analysis of the obtained numerical solutions indicates that the corresponding exact solutions have a finite escape time and verifies the results obtained with the help of Theorem~\ref{Th_Neust1}.

\vspace*{-0.3cm}

 \section{Conclusions}\label{Conclusions}

\vspace*{-0.2cm}

The theorems, enabling to prove the existence and boundedness of global solutions (Lagrange stability) of the semilinear DAE \eqref{DAE} or their absence (solutions have a finite escape time, i.e., they are Lagrange unstable), are obtained.  Using the obtained theorems, we have found the restrictions on the initial data and the parameters of the electrical circuit (Fig.~\ref{MatMod1}) of the nonlinear radio engineering filter under which the mathematical model (the DAE \eqref{DAE_MM}) of the circuit is Lagrange stable, and the conditions under which the mathematical model is Lagrange unstable. The concrete functions and quantities defining the circuits parameters (resistances, conductivities and others) and satisfying the obtained conditions have been given.  It has been checked that the mentioned conditions of the Lagrange stability are fulfilled for certain classes of nonlinear functions, which do not satisfy the global Lipschitz condition.
In particular, it has been proven that the presence of nonlinear resistances and conductivities of the form~\eqref{os6},~\eqref{os7} in electric circuits admits the Lagrange stability of the corresponding mathematical models. Notice that nonlinear resistances and conductivities of such type are often encountered in real radio engineering systems.

The results of the investigation of the mathematical model have shown that a practical check of the conditions of the obtained theorems is sufficiently effective and these conditions can be physically feasible. The analysis of the numerical solutions of the mathematical model verifies the results of theoretical investigations.

\vspace*{0.1cm}

\textbf{Acknowledgements.} Supported in part by the Akhiezer Foundation and by the National Academy of Sciences of Ukraine.

\vspace*{-0.3cm}

\renewcommand{\refname}{\centering\textbf{\normalsize References}}

\bibliographystyle{plainnat}
\bibliography{refFilipk}

\end{document}